\documentclass[11pt,a4paper,reqno]{amsart}

\usepackage{amsmath,amsfonts,amssymb,amsthm,epsfig}

 \allowdisplaybreaks
\numberwithin{equation}{section}

\font\script=rsfs10 at 12pt

\def\L{{\mbox{\script L}\;}}

\def\Prob{{\mathcal P}}

\def\N{\mathbb N}

\def\R{\mathbb R}

\def\res{\mathop{\hbox{\vrule height 7pt width .5pt depth 0pt \vrule height .5pt width 6pt depth 0pt}}\nolimits}
\def\proofof#1{\begin{proof}[#1].}
\def\eps{\varepsilon}
\def\e{\varepsilon}
\def\indic{{\bf 1}}

\def\proofof#1{\begin{proof}[Proof of #1]}

\newcommand\f{f{\L}^d\res\Omega}

\newtheorem{thm}{Theorem}[section]
\newtheorem{corollary}[thm]{Corollary}
\newtheorem{proposition}[thm]{Proposition}
\newtheorem{lemma}[thm]{Lemma}
\newtheorem{definition}[thm]{Definition}
\newtheorem{rem}[thm]{Remark}
\newtheorem{example}[thm]{Example}

\title[Optimum and equilibrium in a transport problem with queue]{Optimum and equilibrium in a transport problem with queue penalization effect}

\author{Gianluca Crippa}\address{G.C.: Dipartimento di Matematica, Universit\`a degli Studi di Parma,
viale G.P.~Usberti 53/A (Campus), 43100 Parma, Italy}
\email{gianluca.crippa@unipr.it}
\author{Chlo\'e Jimenez}\address{C.J.: Laboratoire de Math\'ematiques de Brest, 6, avenue Victor Le Gorgeu, CS 93837, F-29238 Brest Cedex 3, France}\email{chloe.jimenez@univ-brest.fr}
\author{Aldo Pratelli}\address{A.P.: Dipartimento di Matematica, Universit\`a degli Studi di Pavia, via Ferrata 1, 27100 Pavia, Italy}\email{aldo.pratelli@unipv.it}

\begin{document}

\begin{abstract} Consider a distribution of citizens in an urban area in which some services (supermarkets, post offices\ldots) are present. Each citizen, in order to use a service, spends an amount of time which is due both to the travel time to the service and to the queue time waiting in the service. The choice of the service to be used is made by every citizen in order to be served more quickly. Two types of problems can be considered: a global optimization of the total time spent by the citizens of the whole city (we define a global {\em optimum} and we study it with techniques from optimal mass transportation) and an individual optimization, in which each citizen chooses the service trying to minimize just his own time expense (we define the concept of {\em equilibrium} and we study it with techniques from game theory). In this framework we are also able to exhibit two time-dependent strategies (based on the notions of prudence and memory respectively) which converge to the equilibrium.
\end{abstract}

\maketitle

\section{Introduction}\label{s:intro}

In this paper we study an optimization problem coming from the modeling of the behaviour of citizens who every day need to use some services (supermarkets, post offices\ldots) present in a city. We consider a bounded set $\Omega \subset \R^d$ which is the geographic reference for the city and a nonnegative function $f$ with unit integral representing the population density. We fix $k$ points $x_1$, $x_2$, \ldots, $x_k$ at which the services are located. If every service was able to answer immediately to the demand of each customer, obviously every person would choose the one closest to his home. Namely, being $p \geq 1$ a fixed number so that the time spent to cover a distance $\ell$ is given by $\ell^p$, then a citizen living at $x$ would choose the service located at $x_i$ if and only if
$$
| x-x_i |^p = \min_{j=1,\ldots,k} | x - x_j|^p \,.
$$

However, in real world, if a service is crowded because a certain amount of citizens have chosen it, the satisfaction of the demand of the customer is not immediate, and some time has to be spent waiting for it. This time spent in the queue surely depends on the amount of people waiting at the service, but also on the characteristics of the service (for instance, its dimension or the ability of the employees). This can be modeled using $k$ functions $h_1$, $h_2$, \ldots, $h_k$ which express the time to be waited in dependence of the amount of people in the queue. If the amount of citizens choosing the service $x_j$ is $c_j$, for $j = 1, \ldots, k$, then a citizen living at $x$ and going to the service $x_i$ needs to use a total amount of time expressed by the quantity
$$
| x-x_i |^p + h_i(c_i)
$$
in order to reach the service $x_i$ and to have his request fulfilled. In particular, the better service for the customer is not necessarily the closest one. It could be convenient for him to go a bit further away in order to be served in a bigger service, in which the queue is faster.

It is surely apparent from the above discussion that, in this modeling, the decision of each citizen depends on the choices of all the other citizens. Each customer chooses the best service according to the corresponding queue, but the queue itself depends on the choices of all the other customers. A possible attempt to tackle this question is to look for a global minimum. Let us denote by $A_j$ the subset of $\Omega$ consisting of the areas in which the citizens using the service $x_j$ live. The total cost, in term of time spent to be served, needed by the citizens of the whole city is expressed by
$$
\sum_{i=1}^k \int_{A_i} \left[| x-x_i |^p +h_i \left(\int_{A_i}f(x)\, dx \right) \right]  f(x)\, dx \,.
$$
If we agree that the choice of each citizen is done in order to minimize this total cost (for instance, if the major of the city has the right to force the choices of the customers, in order to make the global cost as small as possible), we are led to consider the following minimization problem:
\begin{equation}\label{e:intromin}
\inf_{\substack{(A_i)_{i=1,\ldots,k} \\ \text{partition of $\Omega$}}} \left\{ \sum_{i=1}^k \int_{A_i} \left[| x-x_i |^p +h_i \left(\int_{A_i}f(x)\, dx \right) \right]  f(x)\, dx \right\} \,.
\end{equation}
This type of minimization problem fits very well in the theory of optimal mass transportation. This theory goes back to Monge \cite{monge}, who considered the problem of moving a pile of sand into a hole of the same volume, minimizing the transportation cost. We briefly refer on the Monge problem, its relaxation due to Kantorovich (\cite{kantorovich1} and \cite{kantorovich2}) and the recent progresses on it in Section \ref{TransportSection}. The study of the minimization problem \eqref{e:intromin} is done in Section \ref{ECNS}. We define the concept of {\em optimum}, which is a partition $(A_i)_{i=1,\ldots,k}$ solving \eqref{e:intromin}, we prove results of existence and uniqueness of the optimum (under suitable assumptions on the functions $h_i$) and we give some characterizations of it.

However, this kind of solution of the problem is not very natural, from an economic point of view. Indeed, as pointed out in Example \ref{ex:jap}, it could happen that for some citizens the choice of the service forced by the need of producing an optimum for \eqref{e:intromin} is too expensive. It is more realistic to consider as a good solution of the problem a partition of $\Omega$ in which each citizen is ``satisfied'' with his own choice. By this we mean the following: the citizen looks at the behaviour of all the other citizens, and regarding it as fixed he decides whether his decision has been clever or not. In particular, a citizen living at $x$ is satisfied with his choice of going to the service located at $x_i$ if and only if
\begin{equation}\label{e:introgame}
| x-x_i |^p + h_i(c_i) = \min_{j = 1,\ldots, k} \left\{ | x-x_j |^p + h_j(c_j) \right\} \,,
\end{equation}
where the numbers $c_j$, as before, represent the amount of customers choosing the service $x_j$, for $j = 1,\ldots, k$, but now these quantities are seen as fixed by the citizen who evaluates the correctness of his decision. This is radically different from the previous minimization problem, as we explicitly remark in Section \ref{ss:compar}. The condition expressed in \eqref{e:introgame} paves the way for a connection with game theory (for which good references are \cite{friedman} and \cite{Aubin}): in Section \ref{Individual} we define an {\em equilibrium} for our problem as a partition in which every customer is satisfied with his choice, in the sense specified above. In classical game theory this corresponds to the so-called Nash equilibrium. We are able to prove, under suitable assumptions on the functions $h_i$, existence and uniqueness of an equilibrium, and we also comment on the main differences between equilibrium and optimum.

In Section \ref{s:dyna} we finally consider a dynamical evolution of the problem, looking for convergence to the equilibrium. The general question is the following. Assume that the citizens have some strategy to decide, day by day, the most convenient service, using data which come from the previous days. At each day, they decide where to go, considering the lengths of the queues that have been observed in the previous days. Does this iterative scheme converge to the equilibrium? In Example \ref{nonconv} we show that this is not the case, if the strategy of the customers is too na\"ive. If they just choose every day to minimize the right hand side of \eqref{e:introgame}, in which the time spent on the queues in the various services is simply the one observed the day before, then we could obtain some oscillatory phenomena, with no convergence. However, we propose two different strategies, quite natural at a social level, which lead to convergence to the equilibrium.

The first one (see Section \ref{ss:prudence}) is based on the concept of {\em prudence} (which, in some sense, could also be viewed as laziness or resistance to a change of habits). The decision that every citizen makes is again based on the information on the queues of the previous day, but it is now seen as a stochastic variable (in the language of optimal mass transportation this corresponds to the notion of transport plan, see again Section \ref{TransportSection}). It is ``too risky'' for the customer to change immediately the preferred service, thus we assume that he is only going to change his opinion, in the sense that he will be more tempted by the most convenient service, but he will not immediately ``fully choose'' it. The choice between two services of each customer is modeled by a function with values between $0$ and $1$: value $0$ means that the choice is completely in favour of the first service, while value $1$ means that the choice is completely in favour of the second service. Values between $0$ and $1$ mean that the customer chooses a mixed strategy (this is also common in game theory). This can also be interpreted in statistical terms: if at each point of the city there is a condominium instead of a single citizen, then each person will choose the service in such a way that the statistical distribution of the choices follows the stochastic variable defined above. Assuming that the citizens have ``enough prudence'', we are able to show that the iteration of the stochastic choices converges to a deterministic decision, which is an equilibrium for our problem.

The second strategy (see Section \ref{lastsubs}) relies on the notion of {\em memory}. We go back to a deterministic strategy, but we assume that each customer bases his decision not only on the informations coming from the previous day, but also remembering the queues that have been seen in a certain amount of days in the past. Then the choice for the new day is based on an average of all these informations. If the citizens have ``enough memory'', in the sense that they consider in their choice a sufficiently big amount of days, we are able to show that also this iteration converges to an equilibrium for the problem.

We close this introduction by relating our problem to some other optimization problems studied in the previous literature. The idea for our model came from some results relative to the so-called {\em location problem}, see \cite{BJR}, \cite{J} and \cite{BBSS}. In this problem, we are given a distribution of citizens $f$ in a city $\Omega$, but the location of the services is no more fixed. The question is to optimize the location of the services, in such a way that the transportation cost of the citizens onto the services is minimized. We observe that no queues are considered in this model. Two different approaches to this problem are possible. In the {\em long term planning} all the services are built at the same time, minimizing the average distance the people have to cover to reach the nearest service. In the {\em short term planning}, it is assumed that services are built one by one, minimizing the average distance step by step, taking into account at each stage the location of all the services opened at previous steps. Other related works, regarding the modeling of urban areas and the structure of cities, are \cite{BPSS}, \cite{BS},  \cite{CE2}, \cite{CE}, \cite{CS} and \cite{lucas}.

\medskip

{\bf Acknowledgment.} The authors warmly thank Luigi Ambrosio, Guy Bouchitt\'e, Guillaume Carlier, Fabio Priuli and Filippo Santambrogio for many valuable conversations on the subject of this work. Moreover, they also thank the Scuola Normale Superiore di Pisa and the Mathematics Department of the University of Pavia for the hospitality during the preparation of this paper.


\section{Notation and preliminary results on optimal mass transportation}\label{s:mass}

In this section we introduce the general setup of the paper and the main notation, we review some basic facts about optimal mass transportation (for which the reader can consult \cite{ambrosio} or \cite{villani}) and we establish some preliminary results which will be useful in the sequel. 

\subsection{Main notation of the paper} We fix a bounded Borel set $\Omega \subset \R^d$ and we denote by $\L^d$ the $d$-dimensional Lebesgue measure in $\R^d$. We consider $k$ points $x_1$, $x_2$, \ldots, $x_k$ belonging to $\Omega$, where $k$ is a strictly positive integer; these points represent the location of the services in the city. The density of the citizens is given by an absolutely continuous probability measure $\mu = \f$, where $f : \Omega \to \R$ is a nonnegative function with unit integral. The measure $\mu$ defined in this way will be typically the reference measure in $\Omega$, and we shall also write $f$-a.e.~and $f$-negligible as a shortening of the more precise notation $\mu$-a.e.~and $\mu$-negligible. For $i=1,\ldots,k$ we consider the functions $h_i : [0,1] \to [0,+\infty[$ which encode the amount of time to be waited in dependence of the amount of people using the service. With this we mean that, if the service located at $x_i$ is chosen by an amount $c_i$ of citizens, then the time that will be spent in the queue by each customer is $h_i(c_i)$. Notice that in this modeling we can also include some penalizations for the particular features of the various services, choosing the queue functions in such a way that $h_i(0)>0$: for instance, higher prices or lower quality of the products. We do not specify a priori any condition on the functions $h_i$, but we will rather clarify the assumptions needed in each particular result. For every Borel set $A \subset \Omega$ we consider the indicatrix function $\indic_A$ defined by 
$$
\indic_A (x) = 
\left\{\begin{array}{ll} 1 & \text{if $x \in A$} \\ 0 & \text{if $x \not \in A$.} \end{array} \right.
$$
The set of probability measures on $\Omega$ and $\Omega \times \Omega$ are denoted by $\Prob(\Omega)$ and $\Prob(\Omega \times \Omega)$ respectively. When $\gamma \in \Prob(\Omega)$ is given, we denote by ${\rm L}^1_\gamma (\Omega)$ the vector space consisting of the Borel functions which are $\gamma$-integrable. A partition of $\Omega$ is a finite or countable family $(A_i)_{i \in \N}$ of pairwise disjoint (up to $f$-negligible sets) Borel sets $A_i \subset \Omega$ such that 
$\cup_{i\in \N} A_i $ has full measure in $\Omega$.

\subsection{Optimal mass transportation}\label{TransportSection} In 1781 Monge \cite{monge} raised the problem of transporting a given distribution of mass (a pile of sand, for instance) into another one (a hole, for instance), in such a way that the total work done is minimal. In modern terms the problem can be stated as follows. Given two measures $\mu$, $\nu \in \Prob(\Omega)$ and $p \geq 1$ we consider the minimization problem
$$M_p(\mu,\nu) = \inf \left\{ \left(\int_{\Omega}| x -T(x)|^p \, d\mu(x)\right)^{1/p} \;\; : \;\; T : \Omega \to \Omega \text{ Borel and such that } T_\#  \mu=\nu\right\} \,.$$  
Here we denote by $T_\# \mu \in \Prob(\Omega)$ the push-forward of the measure $\mu$, defined by
$$
\left( T_\# \mu \right) (A) = \mu \big( T^{-1} (A) \big) \qquad \text{for every Borel set $A \subset \Omega$.}
$$
Any map $T$ which is admissible in the above problem is called a transport map from $\mu$ to $\nu$.

It is extremely difficult to attack the Monge problem, mainly due to the fact that it is highly nonlinear. This is why the following relaxed formulation, due to Kantorovich (\cite{kantorovich1} and \cite{kantorovich2}), is of great importance. We denote by $\Pi(\mu,\nu) \subset \Prob(\Omega \times \Omega)$ the set of the probability measures $\gamma$ in $\Omega \times \Omega$ with marginals $\mu$ and $\nu$, i.e.~such that $(\pi_1)_\# \gamma = \mu$ and $(\pi_2)_\# \gamma = \nu$, where we denote by $\pi_1$ the projection on the first component and by $\pi_2$ the projection on the second component. Every element of $\Pi(\mu,\nu)$ is called a transport plan from $\mu$ to $\nu$. We notice that $\Pi(\mu,\nu)$ is always nonempty, since for instance $\mu \otimes \nu \in \Pi(\mu,\nu)$. The relaxed formulation of the Monge problem can be stated as follows:
$$
W_p(\mu,\nu) = \inf \left\{\left(\int_{\Omega \times \Omega}| x-y |^p\, d\gamma(x,y)\right)^{1/p} \; : \; \gamma\in \Pi(\mu,\nu)\right\} \,. 
$$
By considering $\gamma = ({\rm id} \times T)_\# \mu$ for any given transport map $T$ we see that the Kantorovich problem is indeed a relaxation of the Monge problem. It turns out that $W_p$ is a distance on $\Prob(\Omega)$, which metrizes the weak convergence of measures (recall that we are assuming that $\Omega$ is bounded). This distance is called the Wasserstein distance of order $p$.

Existence and uniqueness results for the optimal mass transportation problem has been proved only very recently. We refer to \cite{Brenier} and \cite{ks} for the case $p=2$, to \cite{R} and \cite{GM} for the case $p>1$ and to \cite{CFM}, \cite{EG} and \cite{AP} for the case $p=1$. We summarize these results in the following theorem.

\begin{thm}[Existence and uniqueness of an optimal transport map]\label{exist}
Let $\mu$ and $\nu$ be probability measures in $\Omega$ and fix $p\geq 1$. We assume that $\mu$ is absolutely continuous with respect to the Lebesgue measure $\L^d$. Then $M_p(\mu,\nu) = W_p(\mu,\nu)$ and there exists an optimal transport map from $\mu$ to $\nu$, which is also unique $f$-a.e.~if $p>1$.
\end{thm}

\begin{rem}\label{DifToDisc}
{\rm In the particular case when $\mu \ll \L^d$ and $\nu=\sum_{i \in \N} b_i \delta_{y_i}$ any transport map $T$ is associated to a partition $(B_i)_{i \in \N}$ of $\Omega$ in such a way that 
$$
T(x) = \sum_{i \in \N} y_i\indic_{B_i}(x) \qquad \text{and} \qquad \mu(B_i)=b_i \,.
$$
Conversely, any partition $(B_i)_{i \in \N}$ of $\Omega$ satisfying $\mu(B_i)=b_i$ corresponds to a transport map of the form above.}
\end{rem}

In the situation of the above remark any transport plan can be written as
\begin{equation}\label{defpsi}
\gamma = \sum_{i\in\N} (\psi^i \mu) \otimes \delta_{y_i}
\end{equation}
where the functions $\psi^i : \Omega \to \R^+$ are such that $\sum_{i \in \N} \psi^i(x)=1$ for each $x\in\Omega$ and $\int_\Omega \psi^i\,d\mu=b_i$. With this language $\gamma$ corresponds to a transport map if and only if the functions $\psi^i$ have only the values $0$ and $1$.

Another useful feature of Kantorovich's relaxation is the fact that this new problem admits a dual formulation. 

\begin{thm}[Dual formulation]\label{t:dual}
For $\mu$, $\nu \in \Prob(\Omega)$ the following equality holds:
\[
W^p_p(\mu,\nu) =
\sup \left\{
\int_\Omega u \, d\mu + \int_\Omega v \, d\nu \; : \; 
\begin{array}{c}
u \in {\rm L}^1_\mu(\Omega) \,, \; 
v \in {\rm L}^1_\nu(\Omega) \text{ and } \\
u(x)+v(y) \leq | x-y |^p \text{ for $\mu$-a.e.~$x$ and $\nu$-a.e.~$y$}
\end{array} \right\} \,.
\]
Moreover, there exists an optimal pair $(u,v)$ for this dual formulation.
\end{thm}

In the particular case when $\nu=\sum_{i\in \N} b_i \delta_{y_i}$ is an atomic measure the dual formulation reads as:
\begin{equation}\label{dual}
W^p_p \left(\mu,\sum_{i\in \N} b_i \delta_{y_i} \right) =
\sup \left\{
\int_\Omega u \, d\mu + \sum_{i \in \N} b_i v(y_i) \; : \; 
\begin{array}{c}
u \in {\rm L}^1_\mu(\Omega) \,, \; 
v \in {\rm L}^1_\nu(\Omega) \text{ and } \\
u(x)+v(y_i) \leq | x-y_i |^p \\
\text{ for $\mu$-a.e.~$x$ and every $i \in \N$}
\end{array} \right\} \,.
\end{equation}

\subsection{A preliminary result on the shape of partitions}

We present in this subsection a result of standard flavour relative to the shape of the sets in the partition determined by the optimal transport of a diffuse measure to an atomic measure.

\begin{proposition}\label{cellShape} Let $f : \Omega \to \R$ be a nonnegative function such that $\mu = \f$ is a probability measure, and let  $(y_i)_{i\in\N}$ be a sequence of points belonging to $\Omega$. \\
\indent (i) Let $\nu=\sum_{i\in \N} b_i\delta_{y_i}$ and $(B_i)_{i \in \N}$ be a partition of $\Omega$ such that 
the map $T(x) = \sum_{i \in \N} y_i \indic_{B_i} (x)$ is an optimal transport map from $\mu$ to $\nu$. Let moreover the pair $(u,v)$ be any solution of the dual formulation \eqref{dual}. Then we have
\begin{equation}\label{Bouchi}
u(x) = \inf_{i \in \N} \left\{| x-y_i|^p - v(y_i) \right\} = 
\sum_{i \in \N} \left(| x-y_i|^p - v(y_i) \right )\indic_{B_i}(x) \qquad \text{for $f$-a.e.~$x \in \Omega$.}
\end{equation}
\indent (ii) Let  $(B_i)_{i\in\N}$ be a partition of $\Omega$ and set $b_i = \int_{B_i} f(x) \, dx$, $\nu=\sum_{i\in \N} b_i\delta_{y_i}$ and $T(x) = \sum_{i \in \N} y_i \indic_{B_i} (x)$. Let moreover $u \in {\rm L}^1_\mu (\Omega)$ and $v\in{\rm L}^1_\nu(\Omega)$ be two functions (if any) satisfying condition \eqref{Bouchi}. Then we have that $T$ is optimal for $M_p(\mu,\nu)$ and that the pair $(u,v)$ is optimal for the dual formulation \eqref{dual}.
\end{proposition}

\begin{rem}{\rm We easily see that condition \eqref{Bouchi} implies that for every $i \in \N$ the following equality, intended up to $f$-negligible sets, holds:
\begin{equation}\label{Bouchi2}
B_i = \left\{ x \in \Omega \; : \;  | x-y_i|^p -v(y_i) < | x-y_j|^p -v(y_j) \quad \forall j \not= i \right\} \,.
\end{equation}
This equality precisely describes the shape of the sets in the partition. For instance, in the case $p=1$ the boundaries of the sets $B_i$ are hyperboloids and in the case $p=2$ the cells $B_i$ are polytopes. }
\end{rem}

From the above discussion we easily deduce the following corollary, regarding the uniqueness of the optimal transport map from an absolutely continuous measure to an atomic measure, which will be useful in the sequel.

\begin{corollary}\label{c:uniqdir}
Assume that $\mu \ll \L^d$ and $\nu=\sum_{i\in \N} b_i\delta_{y_i}$ are probability measures in $\Omega$. Then the optimal transport map from $\mu$ to $\nu$ is unique $f$-a.e.~even in the case $p=1$.
\end{corollary}

\begin{proof}
The existence of an optimal transport map is ensured by Theorem \ref{exist}. The point is to show that in this case we also have uniqueness. We fix an optimal couple $(u,v)$ for the dual formulation of the Kantorovich problem \eqref{dual} and we consider an optimal transport map $T$. Recalling that, from Remark \ref{DifToDisc}, every transport map is associated to a partition of $\Omega$, we find a partition $(B_i)_{i \in \N}$ of $\Omega$ in such a way that $T(x) = \sum_{i \in \N} y_i\indic_{B_i}(x)$. But recalling \eqref{Bouchi2} we deduce that the sets $(B_i)_{i \in \N}$ are uniquely determined (up to $f$-negligible sets) by $v$, thus we conclude that the optimal transport map $T$ is unique.
\end{proof}

We close this section proving Proposition \ref{cellShape}.

\begin{proof}[Proof of Proposition \ref{cellShape}]
(i) Using Theorem \ref{t:dual} an the fact that the transport $T$ and the pair $(u,v)$ are optimal we deduce that (see Remark \ref{DifToDisc})
$$ \int_{\Omega} \sum_{i \in \N} \indic_{B_i}(x) | x-y_i|^p f(x) \, dx
= W_p^p(\f, \nu) = \int_\Omega u(x) f(x) \, dx + \sum_{i \in \N} b_i v(y_i)\,.$$
Since $T$ is a transport map from $\mu$ to $\nu$ we obtain
$$ \int_{\Omega}\sum_{i \in \N} \indic_{B_i}(x) | x-y_i|^p f(x) \, dx
= \int_{\Omega}\sum_{i \in \N} \indic_{B_i}(x)( u(x)+v(y_i)) f(x) \, dx \,.$$
Then, recalling that in the dual formulation \eqref{dual} the pair $(u,v)$ is subject the the constraint $u(x) + v(y) \leq |x-y|^p$ for $\mu$-a.e.~$x$ and $\nu$-a.e.~$y$, we get
$$ \sum_{i \in \N} \indic_{B_i}(x) | x-y_i|^p = \sum_{i \in \N} \indic_{B_i}(x)
(u(x)+v(y_i)) \qquad \text{for $f$-a.e.~$x \in \Omega$.} $$
This in particular implies
\begin{equation}\label{e:par} 
u(x)= \sum_{i \in \N} (| x-y_i|^p - v(y_i)) \indic_{B_i}(x) \qquad \text{for $f$-a.e.~$x \in \Omega$.} 
\end{equation}
Using again the constraint of the dual formulation \eqref{dual} we immediately deduce \eqref{Bouchi} from \eqref{e:par}.

(ii) First of all we notice that $T$ is a transport map from $\f$ to $\nu = \sum_{i \in \N} b_i \delta_{y_i}$ and that the pair $(u,v)$ is admissible for the dual formulation \eqref{dual}. Moreover using condition 
\eqref{Bouchi} we obtain
$$
\begin{aligned}
\int_{\Omega} | x-T(x)|^p f(x) \, dx =& \sum_{i \in \N} \int_{B_i} | x-y_i|^p f(x) \, dx 
= \sum_{i \in \N} \int_{B_i} \left[ v(y_i)+u(x) \right] f(x) \, dx \\
=& \int_\Omega v(x) \, d\nu(y) + \int_\Omega u(x) f(x) \, dx \,. 
\end{aligned} 
$$
Recalling Theorem \ref{t:dual} we obtain the validity of (ii).
\end{proof}
%


\section{Existence, uniqueness and characterization of the optimum}\label{ECNS}

\subsection{Formulation of the  problem and definition of the optimum} In this section we consider our problem from a global point of view. The unknown is a partition $(A_i)_{i=1,\ldots,k}$ of $\Omega$, where each $A_i$ represents the urban area where customers choosing the service located at $x_i$ live. The total amount of people living in $A_i$ is given by $c_i = \int_{A_i} f(x) \, dx$. Every citizen living at $x$ and using the service located at $x_i$ is going to spend, in order to be served, a time given by
$$
| x-x_i |^p + h_i(c_i) \,.
$$
In this expression, $| x-x_i |^p$ is the transport time and $h_i(c_i)$ is the time spent waiting in the queue. Consequently, the total time spent by the whole population is given by
\begin{equation}\label{e:tottime} 
\sum_{i=1}^k \int_{A_i} \Big[ | x-x_i |^p + h_i(c_i) \Big] f(x) \, dx \,.
\end{equation}
We want to choose the partition $(A_i)_{i=1,\ldots,k}$ of $\Omega$ in such a way that the quantity in \eqref{e:tottime} is as small as possible; thus we are led to the following minimization problem:
\begin{equation}\label{e:minop}
\inf_{\substack{(A_i)_{i=1,\ldots,k} \\ \text{partition of $\Omega$}}}
\left\{\sum_{i=1}^k \int_{A_i} \left[| x-x_i|^p\ +\ h_i\left(\int_{A_i}f(x)\ dx\right)\right]\ f(x)\ dx \right\}\,.
\end{equation}

\begin{definition}[Optimum]
We say that a partition $(A_i)_{i=1,\ldots,k}$ of $\Omega$ is an {\em optimum} if it is a minimizer for \eqref{e:minop}.
\end{definition}

\subsection{Preliminary considerations} In this subsection we prove two lemmas which will be used in the following. We denote by $S$ the unit simplex in $\R^k$ defined by
$$
S = \left\{ c = (c_1,c_2,\ldots,c_k) \in \R^k \; : \; c_i\geq 0\,,\; \sum_{i=1}^k c_i=1 \right\} \,.
$$ 

\begin{lemma}\label{reform}
The following equality holds:
\[
\inf \eqref{e:minop} = 
\inf_{c \in S} \left\{W_p^p \left(f \L^d \res \Omega, \sum_{i=1}^k c_i \delta_{x_i} \right) + 
\sum_{i=1}^k h_i(c_i)c_i \right\} \,.
\]
\end{lemma}
\begin{proof}
By applying Theorem \ref{exist} we have
\begin{eqnarray}
&&  \hspace{-0.7cm}\phantom{=} \inf_{c \in S} \left\{ W_p^p \left( f\L^d\res\Omega, \sum_{i=1}^k c_i\delta_{x_i} \right) 
+ \sum_{i=1}^k h_i(c_i)c_i \right\}  \nonumber \\
&& \hspace{-0.7cm}= \inf_{c \in S} \left\{ M_p^p \left(f \L^d \res\Omega, \sum_{i=1}^k c_i \delta_{x_i} \right) 
+ \sum_{i=1}^k h_i(c_i)c_i \right\} \label{dep} \\
&& \hspace{-0.7cm}=  \inf_{c \in S}  \left\{  \int_\Omega | x-T(x)|^p f(x)\, dx + \sum_{i=1}^k h_i(c_i)c_i 
\, : \, \text{$T$ transport map from $f \L^d \res \Omega$ to $\sum_{i=1}^k c_i \delta_{x_i}$} \right\} \,. \nonumber
\end{eqnarray}
But recalling Remark \ref{DifToDisc} we can deduce
\[
\inf \eqref{dep}
= \inf_{\substack{c \in S \\ (A_i)_{i=1,\ldots,k} \\ \text{partition of $\Omega$}}}
\left\{ \sum_{i=1}^k \int_{A_i} |x-x_i|^p f(x) \, dx + \sum_{i=1}^k h_i(c_i ) c_i \; : \;
\int_{A_i} f(x) \, dx = c_i \right\} 
= \inf \eqref{e:minop} \,,
\]
thus we obtain the thesis.
\end{proof}

\begin{lemma}\label{lemmaF}
The function $F : S \to \R$ defined by
$$
F(c_1,\ldots,c_k) = W_p^p \left( f \L^d \res \Omega , \sum_{i=1}^k c_i \delta_{x_i} \right)
$$
is continuous and convex.
\end{lemma}
\begin{proof}
The continuity follows directly from the fact that the topology induced by the Wasserstein distance of order $p$ is the weak convergence of measures, as observed in Section \ref{TransportSection}.

Let us now show the convexity. Let $(c_1,\ldots,c_k)$ and $(c'_1,\ldots,c'_k)$ be two elements of $S$ 
and let $t\in[0,1]$. Consider two optimal transport plans 
$$ \gamma \in \Pi \left(f \L^d \res \Omega , \sum_{i=1}^k c_i \delta_{x_i} \right) \qquad \text{and}
\qquad  \gamma' \in \Pi \left(f \L^d \res \Omega , \sum_{i=1}^k c'_i \delta_{x_i} \right) \,. $$
This implies that
$$ t \gamma + (1-t) \gamma' \in \Pi \left(f \L^d \res \Omega , 
\sum_{i=1}^k \big(t c_i + (1-t) c'_i \big) \delta_{x_i} \right) \,, $$
hence we deduce
\[
\begin{split}
F \big( t(c_1,\ldots,c_k) & + (1-t)(c_1',\ldots,c_k') \big) 
\; = \; W_p^p \left( \f , \sum_{i=1}^k \big( tc_i + (1-t) c_i' \big) \delta_{x_i} \right) \\
& \leq  \int_{\Omega \times \Omega} |x-y|^p \, d \big( t \gamma +(1-t ) \gamma' \big) (x,y) \\
& =  t \left(\int_{\Omega \times \Omega} |x-y|^p \, d\gamma(x,y) \right)
+ (1-t) \left( \int_{\Omega \times \Omega} |x-y|^p \, d\gamma' (x,y)\right) \\
& =  t F(c_1,\ldots,c_k) + (1-t) F(c_1',\ldots,c_k') \,.
\end{split}
\]
This precisely means that the map $F$ is convex.
\end{proof}

\subsection{Existence and uniqueness of the optimum} In this subsection we show how to use the reformulation of the minimization problem \eqref{e:minop} obtained in Lemma \ref{reform} and the results on optimal mass transportation collected in Section \ref{TransportSection} to show existence and uniqueness (under suitable assumptions) of an optimum for \eqref{e:minop}. In the following we consider for $i = 1 ,\ldots,k$ the functions $\eta_i : [0,1] \to [0, +\infty[$ defined by
\begin{equation}\label{e:eta}
\eta_i (t) = t h_i(t) \,.
\end{equation}

\begin{proposition}\label{exists}
Assume that the functions $h_i $ are lower semi-continuous. Then:
\begin{itemize}
\item[(i)] there exists an optimum $(\widehat A_i)_{i=1,\ldots,k}$ for \eqref{e:minop};
\item[(ii)] if in addition the maps $\eta_i$ are strictly convex the optimum is unique.
\end{itemize}
\end{proposition}

\begin{proof}
(i) Recalling the definitions of $F$ and $\eta_i$ and using Lemmas \ref{reform} and \ref{lemmaF} we have
$$ \inf \eqref{e:minop} = \inf_{c \in S} \left\{ F(c_1,\ldots,c_k) + \sum_{i=1}^k \eta_i(c_i) \right\} \,. $$  
The expression to be minimized on the right hand side of the above equality is a lower semi-continuous function defined on the non-empty compact set $S$, hence there exists a minimizer $(\widehat c_i)_{i=1,\ldots,k}$. By Theorem \ref{exist} there exists an optimal transport map ${\widehat T}$ from $\f$ to 
$\sum_{i=1}^k {\widehat c_i} \delta_{x_i}$. Thanks to Remark \ref{DifToDisc} the transport 
map $\widehat T$ is associated to a partition $({\widehat A_i})_{i=1,\ldots,k}$ of $\Omega$ which satisfies $\widehat c_i = \int_{{\widehat A_i}} f(x) \, dx$. Let us check that this partition is an optimum for \eqref{e:minop}. From the above considerations we deduce
\[\begin{split}
\inf \eqref{e:minop} 
& =\min_{c \in S} \left\{ F(c_1,\ldots,c_k) + \sum_{i=1}^k \eta_i(c_i) \right\} 
\; =  \; F \left(\widehat c_1,\ldots,\widehat c_k \right) 
+ \sum_{i=1}^k \eta_i\left({\widehat c_i}\right) \\
& =W_p^p \left( \f, \sum_{i=1}^k \widehat c_i \delta_{x_i} \right)
+ \sum_{i=1}^k \eta_i \left({\widehat c_i}\right) \\
& =\sum_{i=1}^k \int_{{\widehat A_i}} | x-x_i|^p  f(x) \, dx 
+ \sum_{i=1}^k \eta_i \left(\int_{{\widehat A_i}} f(x) \, dx \right) \,.
\end{split}\]
This exactly means that $({\widehat A_i})_{i=1,\ldots,k}$ is an optimum for \eqref{e:minop}.

(ii) Assume by contradiction that we have two different optima $(\widehat A_i)_{i=1,\ldots,k}$ and $(\widetilde A_i)_{i=1,\ldots,k}$ for \eqref{e:minop}. We define
$$
\widehat c_i = \int_{\widehat A_i} f(x) \, dx \qquad \text{ and } \qquad 
\widetilde c_i = \int_{\widetilde A_i} f(x) \, dx \qquad \text{ for each $i=1,\ldots,k$.} 
$$
It is immediate to check that $\widehat T(x) = \sum_{i=1}^k x_i \indic_{\widehat A_i} (x)$ is an optimal transport map from $\f$ to $\sum_{i=1}^k \widehat c_i \delta_{x_1}$ and that $\widetilde T(x) = \sum_{i=1}^k x_i \indic_{\widetilde A_i} (x)$ is an optimal transport map from $\f$ to $\sum_{i=1}^k \widetilde c_i \delta_{x_1}$. Recalling the fact that the optimal transport map from a diffuse measure to an atomic measure is unique (see Theorem \ref{exist} for the case $p>1$ and Corollary \ref{c:uniqdir} for the case $p=1$) we deduce that $(\widehat c_i)_{i=1,\ldots,k} \not = (\widetilde c_i)_{i=1,\ldots,k}$. But using again the result in Lemma \ref{reform} we deduce that $(\widehat c_i)_{i=1,\ldots,k}$ and $(\widetilde c_i)_{i=1,\ldots,k}$ are minimizers of the map $S \ni (c_1,\ldots,c_k) \mapsto F(c_1,\ldots,c_k) + \sum_{i=1}^k \eta_i(c_i)$. However, the strict convexity of the maps $\eta_i$, together with the convexity of $F$ shown in Lemma \ref{lemmaF}, implies the uniqueness of the minimizer of this map, and from this contradiction we deduce the thesis.
\end{proof}

\subsection{Characterization of the optimum}

We say that a function $\phi : ]0,1] \to [0, +\infty[$ is differentiable if it is differentiable in the usual sense in the open interval $]0,1[$ and has a left derivative in the point $1$. We also use the convention $0 h_i'(0) = 0$.

\begin{proposition}[Necessary optimality condition] \label{propNC}
Let $(A_i)_{i=1,\ldots,k}$ be an optimum for \eqref{e:minop}. Assume that $h_i$ are differentiable in $]0,1]$ and continuous in $0$. Then for all $i=1,\ldots,k$ the following holds:
\begin{equation}\label{CNS}
\left\{
\begin{array}{l}
A_i = \left\{x \in \Omega \; : \; |x-x_i|^p +h_i(c_i)+ c_ih_i'(c_i) < |x-x_j|^p + h_j(c_j)+ c_jh_j'(c_j) \ \  \forall \, j \not = i \right\}  \\[6pt]
c_i = \int_{A_i} f(x) \, dx \,,
\end{array}
\right.
\end{equation}
where as usual equalities between sets are intended up to $f$-negligible sets.
\end{proposition}

\begin{proof} Consider a partition $(A_r)_{r=1,\ldots,k}$ which is an optimum for \eqref{e:minop} and take an $\L^d$-negligible set $N \subset \Omega$ such that for every $x \not \in N$
\begin{itemize}
\item[(i)] $x$ is a Lebesgue point of $f$;
\item[(ii)] if $x \in A_r$, then $x$ is a point of density $1$ in $A_r$.
\end{itemize}
Fix now two indices $i$ and $j$ in $\{1,\ldots,k\}$ and consider a point $x_0 \in A_i \setminus N$. 
For every $\epsilon>0$ let the partition $(\widetilde A_r)_{r=1,\ldots,k}$ be defined by
$$
\begin{cases}
\widetilde A_i = A_i \setminus B_\eps (x_0) \\
\widetilde A_j = A_j \cup \big( A_i \cap B_\eps (x_0) \big) \\
\widetilde A_r = A_r \; \text{ for all } r \in \{1,\ldots,k\} \setminus \{i,j\} \,.
\end{cases}
$$
Let $c_\eps = \int_{A_i \cap B_\eps (x_0)} f(x) \, dx$. By the optimality of the partition 
$(A_r)_{r=1,\ldots,k}$, comparing its total cost with the total cost of the partition $(\widetilde A_r)_{r=1,\ldots,k}$ we deduce
$$
\begin{aligned}
\sum_{r=1}^k \int_{A_r}  |x-x_r|^p & f(x) \, dx + \sum_{r=1}^k c_r h_r(c_r) 
\; \leq \sum_{\substack{r=1,\ldots,k \\ r \not = i, j}} \int_{A_r} |x-x_r|^p  f(x) \, dx \\
& + \sum_{\substack{r=1,\ldots,k \\ r \not = i, j}} c_r h_r(c_r) 
+  \int_{A_i \backslash B_\eps(x_0)} |x-x_i|^p  f(x) \, dx + (c_i - c_\eps) h_i(c_i-c_\eps) \\
& +  \int_{A_j \cup ( A_i \cap B_\eps (x_0) )} |x-x_j|^p f(x) \, dx +(c_j + c_\eps) h_j(c_j+c_\eps) \,.
\end{aligned}
$$
This leads to
\begin{multline*}
\int_{A_i \cap B_\eps(x_0)} |x-x_i|^p f(x) \, dx + c_i h_i(c_i) +c_jh_j(c_j) \\
\leq  (c_i - c_\eps) h_i(c_i-c_\eps)
+ \int_{A_i \cap B_\eps(x_0)} |x-x_j|^p f(x) \, dx + (c_j + c_\eps) h_j (c_j+c_\eps) \,.
\end{multline*}
We now divide this expression by $\omega_d \eps^d$, where $\omega_d$ is the volume of the unit ball in $\R^d$. We obtain
\begin{multline*}
\frac{1}{\omega_d \eps^d} \int_{A_i \cap B_\eps(x_0)} |x-x_i|^p f(x)\,dx 
+ \frac{c_i}{\omega_d \eps^d} \big[ h_i(c_i)-h_i(c_i-c_\eps) \big] 
+ \frac{c_\eps}{\omega_d \eps^d}  h_i(c_i-c_\eps) \\
\leq \frac{1}{\omega_d \eps^d} \int_{A_i \cap B_\eps(x_0)} |x-x_j|^p f(x) \, dx 
+ \frac{c_j}{\omega_d \eps^d} \big[ h_j(c_j+c_\eps) - h_j(c_j) \big] 
+ \frac{c_\eps}{\omega_d \eps^d} h_j(c_j+c_\eps) \,.
\end{multline*}
Letting $\eps \to 0$ and recalling that $x_0$ satisfies assumptions (i) and (ii) we easily obtain
$$
|x_0 - x_i|^p  f(x_0) + f(x_0) c_i h'_i(c_i) + f(x_0) h_i(c_i) \leq  
|x_0 - x_j|^p  f(x_0) + f(x_0) c_j h'_j(c_j) + f(x_0) h_j(c_j) \,,
$$
that is the desired thesis. \end{proof}

The following lemma will imply that condition (\ref{CNS}) is in fact sufficient, under reasonable assumptions.

\begin{lemma}\label{Unicity} 
Assume that the functions $h_i$ are differentiable in $]0,1]$ and continuous in $0$ and that $\eta_i$ are convex. Then there exists at most one partition $(A_i)_{i=1,\ldots,k}$ for which \eqref{CNS} holds.
\end{lemma}

We notice that the convexity assumption on $\eta_i$ is satisfied for instance if $h_i(t) = t^q$ with $0<q<1$, which are natural concave cost functions.

\begin{proof}[Proof of Lemma \ref{Unicity}]
Assume by contradiction that there exist two partitions $(A_i)_{i=1,\ldots,k}$ and $(\widehat A_i)_{i=1,\ldots,k}$ satisfying \eqref{CNS}. We set $c_i = \int_{A_i} f(x) \, dx$ and $\widehat c_i = \int_{\widehat A_i} f(x) \, dx$ and we consider the set $I = \big\{ i \in \{1,\ldots,k\} \, : \, c_i < \widehat c_i \big\}$. Assume that $I$ is nonempty. Recalling the definition of the functions $\eta_i$ and comparing with \eqref{CNS} we obtain
$$
\bigcup_{i\in I} A_i = \left\{ x \in \Omega \; : \; \min_{i\in I} | x- x_i |^p +\eta'_i (c_i) < 
\min_{j\not \in I} | x-x_j |^p + \eta'_j (c_j) \right\}
$$
and
$$
\bigcup_{i\in I} \widehat A_i = \left\{ x \in \Omega \; : \; \min_{i\in I} | x- x_i |^p + \eta'_i (\widehat c_i) <
\min_{j\not \in I} | x-x_j |^p + \eta'_j (\widehat c_j) \right\} \,.
$$  
Since we are assuming that the functions $\eta_i$ are convex, we obtain that $\eta'_i(c_i) \leq \eta'_i(\widehat c_i)$ for every $i \in I$ and that $\eta'_j(c_j) \geq \eta'_j(\widehat c_j)$ for every $j \not \in I$. This immediately implies that $\cup_{i\in I} \widehat A_i \subseteq \cup_{i\in I} A_i $, thus
\begin{equation}\label{e:bad}
\int_{\cup_{i\in I} \widehat A_i } f(x) \, dx \leq \int_{\cup_{i\in I} A_i} f(x) \, dx \,.
\end{equation}
But recalling the definition of $c_i$ and $\widehat c_i$ and of the set $I$ we have
$$
\int_{\cup_{i\in I} A_i} f(x) \, dx = \sum_{i \in I} c_i < \sum_{i \in I} \widehat c_i = \int_{\cup_{i\in I} \widehat A_i } f(x) \, dx \,,
$$
and this is in contradiction with \eqref{e:bad}. We deduce that the set $I$ is empty, which means that $c_i \geq \widehat c_i$ for every $i = 1 , \ldots, k$. By symmetry the opposite inequality is also true, thus we obtain that $c_i = \widehat c_i$ for every $i = 1 , \ldots, k$. But going back to \eqref{CNS} we immediately deduce that $A_i = \widehat A_i$ for every $i = 1, \ldots, k$.
\end{proof}

\begin{proposition}[Sufficient optimal condition]\label{p:suffop}
Let $(A_i)_{i=1,\ldots,k}$ be a partition of $\Omega$ which satisfies \eqref{CNS}. Assume
that the functions $h_i$ are differentiable in $]0,1]$ and continuous in $0$ and that $\eta_i$ are convex. Then $(A_i)_{i=1,\ldots,k}$ is the unique optimum for \eqref{e:minop}.
\end{proposition}

\begin{proof}
It is a direct consequence of Propositions \ref{exists} and \ref{propNC} and of Lemma \ref{Unicity}.
\end{proof}

\subsection{Summary of the results on the optimum}

We close this section presenting a summary of the results relative to the optimum for \eqref{e:minop}. Remember that the functions $\eta_i$ have been defined in \eqref{e:eta}.
\begin{itemize}
\item If $h_i$ are lower semi-continuous, then there exists an optimum for \eqref{e:minop}.
\item If $h_i$ are lower semi-continuous and $\eta_i$ are strictly convex, then there exists a unique optimum for \eqref{e:minop}.
\item If $h_i$ are differentiable in $]0,1]$ and continuous in $0$ and $\eta_i$ are convex, then there exists a unique optimum for \eqref{e:minop}.
\item If $h_i$ are differentiable and continuous in $0$, then \eqref{CNS} is a necessary optimality condition.
\item If $h_i$ are differentiable in $]0,1]$ and continuous in $0$ and $\eta_i$ are convex, then \eqref{CNS} is a necessary and sufficient optimality condition.
\end{itemize}

We remark in passing what follows: in the general setting of the Monge problem, the ``distance'' $|y-x|^p$ is replaced by any lower semi-continuous function $c(x,y)$, usually called cost function. But in fact we have used only two features of the particular choice $c(x,y)=|y-x|^p$, namely, the existence and uniqueness of an optimal transport map and the fact that for any $y\in\Omega$ the level sets of $c(\cdot,y)$ are negligible. Hence our results can be extended for general costs $c(x,y)$ for which both properties hold true, such as $c(x,y) = |y-x|^p$ with $0<p<1$ (see \cite{GM}) or $c(x,y) = \|y-x\|$ for different kinds of norm (see \cite{CFM} and \cite{cristal}).


\section{Existence, uniqueness and charaterization of the equilibrium}\label{Individual}

\subsection{Definition of the equilibrium}

In this section we start to consider the situation from the point of view of the single citizen. 
Assume for the moment that the sets $A_i$, and so the quantities $c_i$, are given: 
the single citizen is probably not interested whether the partition $(A_i)_{i=1,\ldots,k}$ is optimal in the global 
sense that we discussed up to now. Indeed, it is quite convincing that he does not even know the sets $A_i$, 
nor the quantities $c_i$, nor the functions $h_i$: what is meaningful, is that he only knows the quantities $h_i(c_i)$, i.e.~the queue times in the various services. 
Then, of course, a citizen living at $x\in\Omega$ and going to $x_i$ shall be ``satisfied'' if and only if
$$
|x-x_i|^p+h_i(c_i) = \min_{j=1,\ldots,k} |x-x_j|^p+h_j(c_j)\,.
$$

In the context of game theory this is precisely a Nash equilibrium (see Section 12.3 of \cite{Aubin}): each player is satisfied in the sense that his strategy (i.e.~the choice of the service) is the best possible, once the behaviour of the other players (in this case, the sets $A_i$) is fixed. We give the following definition.

\begin{definition}
We say that a partition $(A_i)_{i=1,\ldots,k}$ of $\Omega$ is an {\em equilibrium} if for every $i=1,\ldots,k$ the following condition holds:
\begin{equation}\label{CE}
\left\{
\begin{array}{l}
A_i = \Big\{x \in \Omega \; : \; |x-x_i|^p +h_i(c_i) < |x-x_j|^p + h_j(c_j) \quad \text{for every $j \not = i$} \Big\} \\ \\
c_i = \int_{A_i} f(x) \, dx \,.
\end{array}
\right.
\end{equation}
\end{definition}

We remark that this type of equilibrium is non-cooperative: each citizen chooses the service just by himself, without collaborating with the other citizens in order to choose a ``better'' global strategy. In Theorem \ref{exEqui} and Proposition \ref{exuninop} we show that there exists an equilibrium, under fairly general assumptions: we remark that this is not a trivial result, since in general game theory it is typically a difficult task to show existence of a Nash equilibrium.

\subsection{Existence and uniqueness of the equilibrium}

In this subsection we show how it is possible to formulate an auxiliary problem in such a way that the optimum for this new problem corresponds to the equilibrium for the original problem. Thus we easily deduce some existence and uniqueness results for the equilibrium, relying on the results presented in the previous section. 

Assume that the functions $h_i$ are continuous. We introduce for $i=1,\ldots,k$ the functions $g_i : [0,1] \to [0,+\infty[$ defined by
$$
g_i(t) = \left\{
\begin{array}{ll}
\displaystyle{1\over t}\int_0^t h_i(s)\ ds & \;\text{if $0 < t \leq 1$} \\ \\
h_i(0) & \;\text{if $t=0$.}
\end{array}\right.
$$
It is immediate that the functions $g_i$ are continuous and that $t \mapsto tg_i(t)$ are differentiable everywhere in $[0,1]$ with derivative $h_i(t)$.

\begin{proposition}\label{linkOpEqu}
Assume that the functions $h_i$ are continuous.
Every partition $(A_i)_{i=1,\ldots,k}$ of $\Omega$ which is a minimizer of the problem
\begin{equation}\label{e:auprob}
\inf_{\substack{(A_i)_{i=1,\ldots,k} \\ \text{{\rm partition of $\Omega$}}}} 
\left\{\sum_{i=1}^k \int_{A_i} \left[| x-x_i|^p + g_i \left(\int_{A_i}f(x)\, dx \right)\right]  f(x) \, dx \right\}
\end{equation}
is an equilibrium for the original problem. If in addition the functions $h_i$ are non-decreasing then every equilibrium for the original problem is a minimizer of \eqref{e:auprob}.
\end{proposition}

\begin{proof}
Consider a minimizer $(A_i)_{i=1,\ldots,k}$ of \eqref{e:auprob}. Since the functions $g_i$ are continuous in $0$ and the functions $t g_i(t)$ are differentiable in $[0,1]$ we can apply Proposition \ref{propNC} to deduce that $(A_i)_{i=1,\ldots,k}$ satisfies the optimality condition \eqref{CNS}, but in view of the definition of $g_i$ this condition is precisely \eqref{CE}: recall that $\big( tg_i(t) \big)' = h_i(t)$. Thus $(A_i)_{i=1,\ldots,k}$ is an equilibrium for the original problem. If $h_i$ is non-decreasing for all $i=1,\ldots, k$, we immediately obtain the convexity of the functions $t \mapsto tg_i(t)$, thus we can apply Proposition \ref{p:suffop}, obtaining precisely that every equilibrium for the original problem is a minimizer of \eqref{e:auprob}.
\end{proof}

\begin{rem}{\rm Arguing similarly to the previous proof, we can also show that, under the assumption that the maps $\eta_i$ defined in \eqref{e:eta} are differentiable in $[0,1]$, any optimum for \eqref{e:minop} is an equilibrium for the problem with queue functions 
$$
\tilde h _i (t) = h_i(t) + t h'_i(t) \qquad  \qquad i=1,\ldots,k \,.
$$}
\end{rem}

Using the correspondence given by the previous proposition, it is now easy to show the following theorem.

\begin{thm}\label{exEqui}
Assume that the functions $h_i$ are continuous. Then there exists an equilibrium. If in addition the functions $h_i$ are non-decreasing then the equilibrium is unique.
\end{thm}

\begin{proof}
Under the assumptions of the theorem we have that the functions $t \mapsto t g_i(t)$ are lower semi-continuous, thus applying Proposition \ref{exists} we obtain the existence of a minimizer of \eqref{e:auprob},  and using the result of Proposition \ref{linkOpEqu} we deduce that this minimizer is an equilibrium for the original problem. If the functions $h_i$ are non-decreasing we apply again Proposition \ref{linkOpEqu}, obtaining that every equilibrium for the original problem is a minimizer of \eqref{e:auprob}. But since $h_i$ are non-decreasing we also deduce that the maps $t \mapsto t g_i(t)$ are convex, thus using Proposition \ref{p:suffop} we deduce that \eqref{e:auprob} has a unique solution, and this concludes the proof.
\end{proof}

\subsection{A direct proof in the case $k=2$.}

We now give an alternative and more direct proof of the above result in the particular case $k=2$. This has also the advantage of introducing some ideas and notation that will be important in the dynamical analysis of Section \ref{s:dyna}. In what follows, we will make extensively use of the following definitions:
\begin{align}\label{e:taum}
\tau(x) = |x-x_1|^p - |x-x_2|^p\,,  && m(t) = \int_{\{x \,:\, \tau(x) < t\}} f(x)\, dx\,.
\end{align}
The condition in \eqref{CE} defining the equilibrium can be read in this context as follows: {\em a partition $(A_1,A_2)$ of $\Omega$ is an equilibrium when there exists $\bar t\in \R$ such that:
\begin{equation}
\label{formastandard}
A_1 = \big\{x\in\Omega \,:\, \tau(x) < \bar t \big\},\quad
A_2 = \big\{x\in\Omega \,:\, \tau(x) > \bar t \big\}
\end{equation}
\begin{equation}\label{optind}
h_2\big(1-m(\bar t)\big) - h_1\big(m(\bar t)\big) = \bar t\,.
\end{equation}
}
\begin{proposition}\label{exuninop}
Assume that the functions $h_1$ and $h_2$ are continuous and non-decreasing. 
Then there exists a unique equilibrium $(A_1,A_2)$.
\end{proposition}

\begin{proof}
Let us consider the map 
$$
U (t) = t-h_2 \big(1-m(t) \big) + h_1 \big(m(t)\big) \,.
$$
It is a continuous and strictly increasing map. Notice that for $t\ge \sup_\Omega \tau$ we have $m(t)=1$ and $U(t)=t-h_2(0)+h_1(1)$ so that $\lim_{t\to +\infty} U(t)=+\infty$. In the same way for $t\le \inf_\Omega \tau$ we have $m(t)=0$ and $U(t)=t-h_2(1)+h_1(0)$, thus $\lim_{t\to -\infty} U(t)=-\infty$. By consequence, there exists a unique ${\bar t}$ such that $U(\bar t)=0$, that is, \eqref{optind} holds. Then the partition $(A_1,A_2)$ associated to $\bar t$ as in \eqref{formastandard} is an equilibrium and this is unique.
\end{proof}

\subsection{A comparison between optimum and equilibrium}\label{ss:compar}

Recall that the optimality condition for the optimum obtained in Proposition \ref{propNC}, read with the notation of the previous subsection, was
$$
 h_2 \left(1-m \left(\hat t \right) \right) + \left(1-m \left(\hat t \right)\right)  h'_2 \left(1-m \left(\hat t \right)\right)- h_1\left(m \left(\hat t \right)\right) -m \left(\hat t\right) h'_1 \left(m \left(\hat t \right) \right)
= \hat t \,.
$$
Thus a comparison with \eqref{optind} shows how deep is the difference between the two conditions of optimality. In this subsection we give two explicit examples in which we see in a qualitative way some differences between the optimum and the equilibrium.

\begin{example}{\rm 
On a beach represented by $\Omega=[0,1]$ there are two ice-cream shops at coordinates $x_1=1/4$ and $x_2=3/4$. Suppose 
there are less employes in the second shop, so that $h_1(t)=t$ and $h_2(t)=(1+\e)t$. Assume that the costumers are uniformly distributed
on the beach, that is $f \equiv 1$, and take $p=2$.

In this case the optimum $(A_1,A_2)$ is given by
$$
A_1=[0,\lambda_\e^{\rm opt}[\,,\quad A_2=]\lambda_\e^{\rm opt}, 1] \qquad \mbox{ with } \lambda_\e^{\rm opt}={1\over 2} + {\e\over 5+2\e} \,,
$$
whereas the equilibrium $(B_1,B_2)$ is
$$
B_1 = [0,\lambda_\e^{\rm eq}[ \,, \quad B_2 = ]\lambda_\e^{\rm eq}, 1] \qquad \mbox{ with } \lambda_\e^{\rm eq}={1\over 2} + {\e\over 6 + 2\e} \leq \lambda_\e^{\rm opt}\,.
$$
The costs for the customers in the two cases are given by:
\begin{equation*}
\begin{array}{|c|c|c|}
\hline & &\\[-10pt]
\begin{array}{c} \text{cost on the}\\ \text{optimum} \end{array} & \begin{array}{c} \text{cost on the}\\ \text{equilibrium} \end{array} & \text{position} \\[10pt]
\hline && \\
| x-1/4 |^2 + \lambda_\e^{\rm opt} & | x-1/4 |^2 + \lambda_\e^{\rm eq} & 0 \leq x \leq \lambda_\e^{\rm eq} \\[10pt]
| x-1/4 |^2 + \lambda_\e^{\rm opt} &  | x-3/4 |^2 +(1+\e)(1-\lambda_\e^{\rm eq}) & \lambda_\e^{\rm eq} \leq x \leq \lambda_\e^{\rm opt} \\[10pt]
| x-3/4 |^2 + (1+\e)(1-\lambda_\e^{\rm opt}) &  | x-3/4 |^2 + (1+\e)(1-\lambda_\e^{\rm eq}) & \lambda_\e^{\rm opt} \leq x \leq 1 \\[5pt]
\hline
\end{array}
\vspace{8pt}
\end{equation*}
Note that, in the case $\e=0$, the situations of optimum and equilibrium coincides and we have 
$A_1 = B_1 = [0,1/2[$ and $A_2 = B_2 = ]1/2,1]$. In the case $\e>0$, despite
the global cost is minimized by the partition $(A_1,A_2)$, in the equilibrium situation more than one half of the costumers pay less than in the optimum situation.}
\end{example}

\begin{example}\label{ex:jap}
{\rm  Let us now take $\Omega=[0,1]$, $f \equiv 1$, $x_1=0$, $x_2=1$, $p=1$ and 
$$
h_1(s) \equiv 100  \qquad \text{ and } \qquad
h_2(s) = \left\{\begin{array}{ll}
0 & \text{ for $0\leq s\le 0.999$} \\
1 & \text{ for $0.999 < s \leq 1$.}
\end{array}\right.
$$
It is clear that the unique equilibrium is $A_1=\emptyset$ and $A_2=[0,1]$, whereas the optimum is obtained with $A_1=[0,0.001[$ and $A_2= ]0.001,1]$. Notice that the optimum is very unfair for costumers living in $A_1$, who pay $100+x$, whereas the other costumers just pay the distance 
from $1$. Note that the result would have the same features with $h_2$ equal to $0$ on $[0,0.998]$, equal to $1$ on $[0.999,1]$ and increasing smoothly from $0$ to $1$ on $]0.998,0.999[$. }
\end{example}

\subsection{A comparison with Pareto optimum}

The examples given in the previous subsection show that in some cases the global optimum can really be not convenient for some citizens; in particular, in Example~\ref{ex:jap} the optimum is better than the equilibrium for $99.9\%$ of the citizens, even though it is much worse for the remaining $0.1\%$. In some classical games (like the well-known prisoner's dilemma, see for instance Section 7.8 of \cite{Aubin}) the equilibrium is in fact a bad strategy for \emph{all} players, in the sense that there is a situation which is better for everybody. 

On the contrary, we are going to see that in our model this can not happen. Even more, given an equilibrium, we show that it is not possible to find a situation in which every citizen spend less time. More precisely, and introducing a further notion from game theory, we are going to see that every equilibrium is a Pareto optimum for the problem. This means that, in our situation, starting from the non-cooperative Nash equilibrium it is not possible to lower the costs of all the citizens by a cooperation in the choice of the services.

Let us introduce the individual cost function
\begin{equation}\label{e:C}
C(x,(B_i)_i) = \sum_{i=1}^k \left[ | x-x_i |^p +h_i\left(\int_{B_i} f(x) \,dx\right) \right] \indic_{B_i}(x) \,,
\end{equation}
defined for $x\in\Omega$ and $(B_i)_{i=1,\ldots,k}$ partition of $\Omega$.  We say that a partition $(A_i)_{i=1,\ldots, k}$ of $\Omega$ is a {\em Pareto optimum} (see also Section 12.5 of \cite{Aubin}) if there exists no partition $(B_i)_{i=1,\ldots,k}$ of $\Omega$ satisfying
\[
C(x,(B_i)_i) \leq C(x,(A_i)_i) \qquad  \text{for $f$-a.e.~$x \in \Omega$}
\]
with strict inequality in a set of strictly positive measure. With this definition of $C$ the following equivalence is immediate:
\begin{equation}\label{equivalence}
(A_i)_{i=1,\ldots,k}\, \hbox{is an equilibrium} 
\, \Longleftrightarrow \,
C(x,(A_i)_i) = \min_{j=1,\ldots,k} \bigg\{ | x-x_j |^p + h_j \bigg(\int_{A_j} f(x) \, dx \bigg) \bigg\} \,.
\end{equation}

We establish now the following result of independent interest, which makes a connection between the equilibrium and the optimal transport map from $\mu=\f$ to the atomic measure concentrated on the points $x_1,\ldots,x_k$.

\begin{proposition}[Link between equilibrium and optimal transportation] \label{LTE}
Let $(A_i)_{i=1,\ldots,k}$ be a partition of $\Omega$, $\mu=\f$, $c_i=\int_{A_i} f$, $\nu=\sum_{i=1}^k c_i\delta_{x_i}$ and $T=\sum_{i=1}^k x_i\indic_{A_i}$; let moreover $u(x)= C(x,(A_i)_i)$ and $v(x_i)=-h_i(c_i)$. Then the couple $(u,v)$ is optimal for the dual problem~(\ref{dual}) if and only if $(A_i)_{i=1,\ldots,k}$ is an equilibrium. Moreover, in this case $T$ is an optimal transport map between $\mu$ and $\nu$.
\end{proposition}
\begin{proof}
First of all, one can notice that the cost of the transport $T$ equals $\int_\Omega u \,d\mu+ \sum_i c_i v(x_i)$: indeed, one has
\[\begin{split}
\int_\Omega |x-T(x)|^p\,d\mu &= \sum_{i=1}^k \int_{A_i} | x-x_i |^p f(x) \, dx
=\sum_{i=1}^k \int_{A_i} \Big[|x-x_i|^p +h_i(c_i)\Big]f(x) \, dx + \sum_{i=1}^k c_i v(x_i)\\
&= \int_\Omega u(x)\,d\mu(x) + \int_\Omega v(x)\,d\nu(x) \,.
\end{split}\]
Thanks to Theorem~\ref{t:dual} we deduce that, if $(u,v)$ is admissible for the dual problem, then $T$ is an optimal transport map and $(u,v)$ is optimal for the dual problem.\par
On the other hand, the property~(\ref{equivalence}) exactly means that $(u,v)$ is admissible for the dual problem if and only if $(A_i)_{i=1,\ldots,k}$ is an equilibrium. These two considerations give the thesis.
\end{proof}
Notice that, in the situation of the proposition above, it may happen that $T$ is an optimal map but $(u,v)$ is \emph{not} optimal for the dual problem and $(A_i)_{i=1,\ldots,k}$ is \emph{not} an equilibrium.\par
We now use Proposition~\ref{LTE} to prove that an equilibrium is a Pareto optimum.
\begin{proposition}\label{Pareto}
Assume that the maps $h_i$ are strictly increasing and let $(A_i)_{i=1,\ldots,k}$ be an equilibrium. Let $(B_i)_{i=1,\ldots,k}$ be a partition of $\Omega$ such that
\begin{equation}\label{e:cpar}
C(x, (B_i)_i) \le C(x,(A_i)_i) \qquad \qquad \text{for $f$-a.e.~$x \in \Omega$.}
\end{equation}
Then the partitions $(A_i)_{i=1,\ldots,k}$ and $(B_i)_{i=1,\ldots,k}$ coincide.
\end{proposition}
\begin{proof}
We first show that for every $i=1,\ldots,k$ the equality $\int_{A_i}f(x) \, dx = \int_{B_i} f(x) \,dx$ holds. Indeed, if this were not true, we could find $j \in \{1,\ldots,k\}$ such that $\int_{A_j}f(x) \, dx < \int_{B_j} f(x) \,dx$. Considering a point $x\in B_j$ such that \eqref{e:cpar} holds we have
\[\begin{split}
| x-x_j |^p +h_j \left(\int_{B_j} f(x) \,dx \right) &= C(x,(B_i)_i) \le C(x,(A_i)_i)\\
& =\min_{i=1,\ldots,k} \left\{ | x-x_i |^p +h_i\left(\int_{A_i} f(x) \, dx \right)\right\}\\
&\leq | x-x_j |^p +h_j \left(\int_{A_j} f(x) \, dx \right) \,.
\end{split}\]
But this leads to a contradiction: indeed, $h_j$ being strictly increasing, we should have 
$h_j \left(\int_{A_j} f(x) \, dx \right) <  h_j \left( \int_{B_j}  f(x) \, dx \right)$. This shows as claimed that
$\int_{A_i}f(x) \, dx = \int_{B_i} f(x) \,dx$ for every $i=1,\ldots,k$.

Since $(A_i)_{i=1,\ldots,k}$ is an equilibrium, using Proposition \ref{LTE} we obtain that the map $T=\sum_{i=1}^k x_i\indic_{A_i}$ is an optimal transport map from $\mu$ to $\sum_{i=1}^k \left(\int_{A_i} f(x)\, dx\right)\delta_{x_i}$. Using this fact and applying the assumption \eqref{e:cpar} we deduce
\[\begin{split}
\sum_{i=1}^k \int_{B_i} | x-x_i |^p f(x) \, dx
&\leq \sum_{i=1}^k \int_{A_i} | x-x_i |^p f(x) \, dx \\
& = W_p \bigg(\mu,\sum_{i=1}^k \bigg( \int_{A_i} f(x) \, dx \bigg) \delta_{x_i} \bigg) = 
W_p \bigg(\mu,\sum_{i=1}^k \bigg(\int_{B_i}f(x) \, dx\bigg) \delta_{x_i} \bigg) \,.
\end{split}\]
Hence the map  $S=\sum_{i=1}^k x_i \indic_{B_i}$ is an optimal transport map from $\mu$ to $\sum_{i=1}^k \left(\int_{B_i} f(x)\, dx\right)\delta_{x_i} = \sum_{i=1}^k \left(\int_{A_i} f(x)\, dx\right)\delta_{x_i}$. By the uniqueness results of Theorem \ref{exist} and Corollary \ref{c:uniqdir} we conclude.
\end{proof}

\begin{corollary}
If the maps $h_i$ are strictly increasing, then every equilibrium is a Pareto optimum.
\end{corollary}



\section{Evolution dynamic  and convergence to the equilibrium}\label{s:dyna}

In this section we discuss a dynamical formulation of the problem, in which every single day each citizen decides where to go, using the knowledge of what happened in the previous days and trying to make a ``smart'' choice. We set the problem and discuss under which assumptions the situation converges to the (unique) individual equilibrium, giving also counterexamples to show that this is not always the case. In the first subsection we consider the standard evolution, where every day one thinks only to the previous day. In the second subsection we study the evolution with prudence, where the citizens again remember only the previous day, but are more careful and try to avoid changing idea too often. In the last subsection we describe the case of the evolution with memory, where each day the citizens remember more previous days. In all this section, to keep the discussion as simple as possible, we consider the case $k=2$ when there are only two services located at $x_1$ and $x_2$; however, the general case requires more care with multiple indeces but no really new ideas.\par
Recall that, thanks to Proposition \ref{exuninop}, we know the existence of a unique individual equilibrium under the assumption that the functions $h_1$ and $h_2$ are continuous and non-decreasing.

\subsection{Standard evolution\label{secstand}}

We face now the question whether or not the situation can naturally evolve toward this equilibrium. By ``naturally'' we mean that we try to model the evolution in time, in which the citizens do not know the situation in its complete complexity, but they just see how long is the queue and decide day by day how to behave. In other words, we introduce a simple scheme which models the everyday choice of people, and we study the possible convergence under this model. The idea is very easy: at each day, each citizen decides freely whether to go to $x_1$ or to $x_2$. Moreover, at the end of every day he discovers what has happened in that day, that is, he discovers $h_1(c_1)$ and $h_2(c_2)$: this will clearly affect the choice of the following day. More precisely, for each $j\in\N$ we have the function $\psi_j:\Omega\to \{0,1\}$, where $\psi_j(x)=0$ (resp. $\psi_j(x)=1$) means that at the day $j$ the citizen living at $x$ goes to $x_1$ (resp. to $x_2$); note that, in the language of \eqref{defpsi}, the function $\psi_j$ corresponds to $\psi^2$ at time-step $j$, while $\psi^1$ is simply $1-\psi^2$. We also set $m_j\in [0,1]$ the mass of the people that, at the day $j$, decide to go to $x_1$, i.e.
\begin{equation}\label{defmj}
m_j:= \int_\Omega \big(1-\psi_j(x)\big)f(x)\,dx\,.
\end{equation}
The evolution of the problem is the following: we fix any function $\psi_0 : \Omega \to \{0,1\}$, and consider it as a data of the problem; that is, we leave the citizens free to decide by chance what to do at the day $0$. Of course, the most meaningful choice would be
\begin{equation}\label{meaningfulchoice}
\psi_0(x)=\left\{\begin{array}{ll}
0 &\hbox{if $\tau(x)<0$}\\
1 &\hbox{otherwise,}
\end{array}\right.
\end{equation}
recalling the definitions of $\tau$ and $m$ given previously
\begin{align}\tag{\ref{e:taum}}
\tau(x) = |x-x_1|^p - |x-x_2|^p\,,  && m(t) = \int_{\{x \,:\, \tau(x) < t\}} f(x)\, dx\,.
\end{align}
Indeed, this choice of $\psi_0$ means that in the first day each one simply goes to the closest of the two service places, which makes sense since he does not have any information about queues. However, our results also hold with any other function $\psi_0$. Given now any $j\in\N$, and given the function $\psi_j$, the function $\psi_{j+1}$ is determined as follows:
\begin{equation}\label{firstevolution}
\psi_{j+1}(x)=\left\{\begin{array}{ll}
0 &\hbox{if $\tau(x)<h_2(1-m_j)-h_1(m_j)$}\\
1 &\hbox{otherwise.}
\end{array}\right.
\end{equation}
The meaning of this model is extremely clear: at the day $j+1$, each citizen decides where to go assuming that the quantity of people going to $x_1$ and $x_2$ will be the same as in the day $j$. The first property that we can easily notice is the following: there exists a sequence $( t_j )_{j\geq 1}$ such that for any $j$ one has
\[
\psi_j(x)=\left\{\begin{array}{ll}
0 &\hbox{if $\tau(x)<t_j$}\\
1 &\hbox{otherwise;}
\end{array}\right.
\]
this is immediate by the definition for any $j\geq 1$. Being $\psi_0$ free, we can not say the same for $j=0$, of course: however, in the case of the ``meaningful'' choice~(\ref{meaningfulchoice}), the above formula holds also for $j=0$ with $t_0=0$. Notice that
\begin{equation}\label{evoluz}
t_{j+1}=h_2(1-m_j)-h_1(m_j);
\end{equation}
as a consequence, we understand that everything in the choices of the citizen living at $x$ depends only on the value of $\tau(x)$. This also implies, for $j\geq 1$, that
\[
m_j=m(t_j)\,.
\]
What we want to know is if the situation converges to the individual optimum; that is, if $\psi_j$ converges to the function $\bar\psi$ defined as
\[
\bar \psi (x) = \left\{
\begin{array}{ll}
0 & \text{ if $x \in A_1$} \\
1 & \text{ if $x \in A_2$,}
\end{array} \right.
\]
where $(A_1,A_2)$ is the equilibrium partition given by Proposition~\ref{exuninop}. Equivalently, we want to understand if $t_j$ converges to $\bar t$ where, according to Proposition~\ref{exuninop}, $\bar t$ is the unique real number such that
\[
\bar t = h_2\big( 1- m(\bar t)\big) - h_1\big(m(\bar t)\big)\,.
\]
A first answer to this question is given below. We shall consider the function $G:\R\to\R$ defined by
\begin{equation}\label{defG}
G(t):= h_2\big(1-m(t)\big)-h_1\big(m(t)\big)\,.
\end{equation}
Notice that the function $G$ has a very simple meaning, namely, it is $t_{j+1}=G(t_j)$; moreover $\bar t$ is uniquely determined by the fact that $G(\bar t)=\bar t$. We then prove the following result which clarifies what happens for the standard evolution.
\begin{proposition}\label{g'<1}
Assume that $h_1$ and $h_2$ are continuous and non-decreasing, as well as that
\begin{equation}\label{condder<1}
| G(s) - G(t) | < |s-t|
\end{equation}
for all $s\neq t\in \R$. Then $t_j\to \bar t$ for $j\to +\infty$, hence also $\psi_j\to \bar\psi$ uniformly on the compact subsets of $\Omega\setminus \{\tau=\bar t\}$.
\end{proposition}
\begin{proof}
The evolution of the sequence $j\mapsto t_j$ is shown in the Figure~\ref{chloegianluca}: notice that since $h_1$ and $h_2$ are non-decreasing then the function $G$ is non-increasing, hence the equilibrium point $\bar t=G(\bar t)$ is unique. 
\begin{figure}[htbp]
\begin{center}
\input{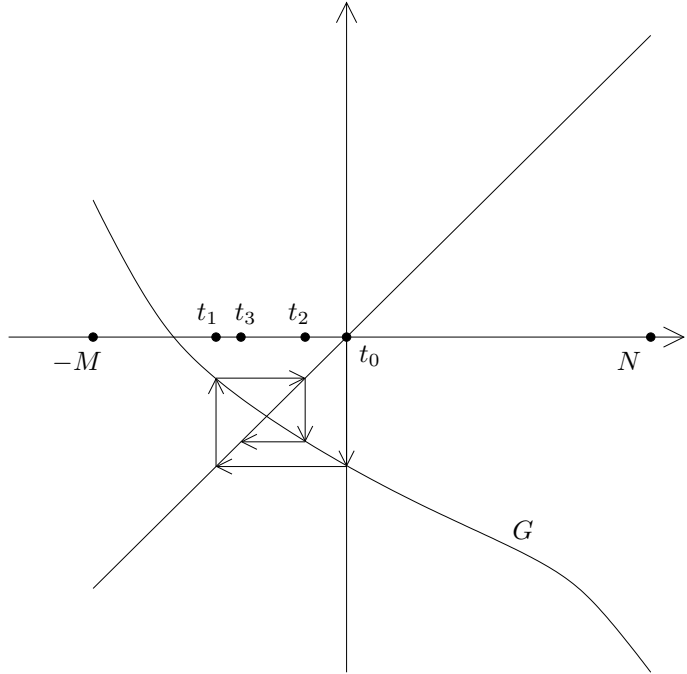}\caption{The situation in Proposition~\ref{g'<1}}\label{chloegianluca}
\end{center}
\end{figure}
Let us assume by symmetry that, as in the figure, $t_1\leq \bar t$: then, since $G$ is a non-increasing function and $G(\bar t)=\bar t$, we have immediately that $t_n\geq\bar t$ for all even $n$ and $t_n\leq \bar t$ for all odd $n$. Therefore, $t_{2n}\searrow t^+$ and $t_{2n+1}\nearrow t^-$ with $t^-\leq \bar t\leq t^+$. Moreover, by continuity of $G$ it is clear that $G(t^+)=t^-$ and $G(t^-)=t^+$. By~(\ref{condder<1}) we deduce that, if $t_-<t_+$, then
\[
t^+ - t^-= G(t^-)-G(t^+) < t^+-t^-\,,
\]
which gives an absurd: hence, $t^-=t^+$, so $t_j\to \bar t$ as we stated. Consider now any compact subset $K$ of $\Omega\setminus \{\tau=\bar t\}$: there is then a small interval $I\ni \bar t$ so that $\tau(x)\notin I$ for all $x\in K$. There exists $\bar j$, depending on $I$ thus on $K$, such that $t_j\in I$ for all $j\geq \bar j$; hence, for any $x\in K$ and for any $j\geq \bar j$, one has $\psi_j(x)=\bar \psi(x)$.
\end{proof}

\begin{rem}
{\rm As the proof underlines, we have more than the uniform convergence of $\psi_j$ to $\bar \psi$ on the compact subsets of $\Omega\setminus\{\tau=\bar t\}$: indeed, in such compact subsets one has that $\psi_j\equiv \bar\psi$ for $j\gg 1$.}
\end{rem}

\begin{example}[Non-convergence]\label{nonconv} {\rm We show now that the \emph{strict} $1$-Lipschitzianity of $G$ required in Proposition~\ref{g'<1} is necessary. Indeed, let $f$ and $h_1$, $h_2$ be such that $G(t)=1-t$ for $t\in (-2,2)$: this is of course possible with a smooth choice of $f$ and of $h_1$ and $h_2$. Moreover, let $\psi_0$ be given by~(\ref{meaningfulchoice}): in this case it is clear that $t_j=0$ for all even $j$ and $t_j=1$ for all odd $j$, so there is not convergence to $\bar t=1/2$. Notice that in this case $G$ is a Lipschitz function with Lipschitz constant equal to $1$.\par
We can easily understand a practical meaning of the possible non-convergence that we have described: if a small town has two supermarkets, there is not a big difference for the people to drive to one or to the other one. However, it may happen that these supermarkets can have very long queues. Therefore, if in a certain day everybody decides to go to the first one, all the people will be unsatisfied with their choice, noticing that the other one has no queue. Consequently, in the next day everybody will go to the second supermarket, leading to a long queue in this one and no queue in the first one. Clearly this procedure will be repeated all the following days, and so there will be no convergence.}
\end{example}

\begin{rem}
{\rm Let us underline what is the meaning of the request of $1$-Lipschitzianity of $G$; if everything is smooth, we can evaluate
\[
G'(t)= - \Big(h_2'(1-m(t)) - h_1'(m(t)) \Big) m'(t)\,;
\]
hence, requiring that $|G'|$ is small means that we want $h_i'$ and/or $m'$ small: the first fact means that the queue does not change dramatically for a small change in the quantity of people; and the second one means that $f$ is not too big, so that a small region of the city (the region where the people can have the doubt whether going to one or the other place) does not contain too many citizens.}
\end{rem}

\subsection{Evolution with prudence}\label{ss:prudence}

As Example~\ref{nonconv} shows, in the previous setting the convergence to the individual optimum does not always hold; however, this does not seem to be likely. In particular, in the case described in the second part of the example, it seems not so reasonable that people continue to change their idea every day in such a silly way. To model the ``smarter'' behaviour of the people, we can follow two ways. One way is to allow people to remember more than a single day, giving them a memory, and this is the content of Section~\ref{lastsubs}. The other way, which is what we investigate now, is to make people more ``prudent''.\par
To introduce the concept of prudence, we start by noticing that the easiness with which people change their habit is excessive: in other words, it is not very reasonable that for a generic $x\in\Omega$ the function $j\mapsto \psi_j(x)$ can change from $0$ to $1$ and vice-versa. The idea of ``prudence'' is that it can be more convenient to consider the choice of a citizen living at $x$ at time $j$ as a stochastic variable, hence using the approach with transport plans instead of transport maps; recalling \eqref{defpsi}, this means that the value of $\psi_j(x)$ can be not only $0$ or $1$, but any number between $0$ and $1$. This approach is also common in game theory, and is related to the notion of mixed strategy, for which we refer to Section 7.4 of \cite{Aubin}. So, $\psi_j(x)$ is the percentage of the citizens living at $x$ which decide to go to $x_2$ at the time $j$, while $1-\psi_j(x)$ is the percentage of those going to $x_1$. With this approach, one is not assuming that all the citizens living at the same place $x$ are forced to take always exactly the same decisions. The evolution rule will be then a simple modification of \eqref{firstevolution}. Recall the definition \eqref{defmj} of $m_j$ which corresponds to the mass of people going to $x_1$ at time $j$. Notice also that the definition \eqref{evoluz} of $t_{j+1}$ means that at time $j+1$, for a citizen living at $x$ it would seem at first glance more convenient to go to $x_1$ (resp. $x_2$) if $\tau(x)$ is smaller (resp. greater) than $t_{j+1}$. We introduce then a parameter $0<\rho<1$ which we call {\em prudence parameter}, expressing the ``resistance'' that the average citizen feels against changing his habit. The evolution rule is then given as follows:
\begin{equation}\label{evolutionprudence}
\psi_{j+1}(x)=\left\{\begin{array}{ll}
\rho \psi_j(x) &\hbox{if $\tau(x)<t_{j+1}$}\\
1-\rho\big(1-\psi_j(x)\big) &\hbox{otherwise.}
\end{array}\right.
\end{equation}
Note that the extreme case $\rho=0$, that is no prudence, corresponds to the case of the previous section; on the opposite, 
the extreme case $\rho=1$, that is full prudence, means that nobody will ever change his habit, and so $\psi_j=\psi_0$ for 
each $j$. We can then prove our result.
\begin{thm}\label{thmprudence}
Assume that $f$ is bounded and that $h_1$ and $h_2$ are non-decreasing and Lipschitz continuous. Then there exists $0<\bar\rho<1$ so that, if $\bar\rho<\rho<1$, $\psi_j$ converges to $\bar\psi$ uniformly on the compact subsets of $\Omega\setminus\{\tau=\bar t\}$.
\end{thm}
\begin{proof}
First of all, we introduce the quantities
\begin{equation}\label{defs_12}
S^2_j = \int_{\{\tau(x)< \bar t\}} \psi_j(x)\, d\mu(x)
\qquad \text{and} \qquad
S^1_j = \int_{\{\tau(x)> \bar t\}} \big(1-\psi_j(x)\big)\, d\mu(x)\,.
\end{equation}
Observe that, at time $j$, $S^2_j$ is the amount of people which are going to $x_2$, but which should go to $x_1$ according to the equilibrium $\bar\psi$; analogously, $S^1_j$ are the people going to $x_1$ and which should go to $x_2$. Clearly, if we call $\bar m=m(\bar t)$, it is $m_j=\bar m+S^1_j-S^2_j$, where $m_j$ is defined in \eqref{defmj}. Hence, comparing~(\ref{optind}) and~(\ref{evoluz}) and recalling that $h_1$ and $h_2$ are Lipschitz, we deduce
\begin{equation}\label{stimatS}
\big|t_{j+1}-\bar t\big| \leq K \big|S^1_j-S^2_j\big|
\end{equation}
with a constant $K$ depending only on $h_1$ and $h_2$. Moreover, since $h_1$ and $h_2$ are non-decreasing and since $G(\bar t)=\bar t$ by definition of equilibrium, we deduce that $m_j \leq \bar m$ implies $t_{j+1} \geq \bar t$ and symmetrically $m_j \geq \bar m$ implies $t_{j+1} \leq \bar t$. This means that if at time $j$ ``few'' people are going to $x_1$, at time $j+1$ ``many'' people will assume to be convenient to do so: however, thanks to the prudence, only a part of those people will indeed change their goal point.\par
Let us now take $j\in \N$ and let us assume by symmetry that $S^1_j\geq S^2_j$, hence $m_j\geq m(\bar t)$, $t_{j+1}\leq \bar t$ and $m(t_{j+1})\leq m(\bar t)$. Keep in mind that, due to the prudence, it is not true in general that $m_j=m(t_j)$. By construction, we know
\begin{equation}\label{S1j+1}
S^1_{j+1} = \rho S^1_j\,.
\end{equation}
On the other hand, concerning $S^2_{j+1}$, by definition~(\ref{defs_12}) one has
\begin{equation}\label{S2j+1}
\begin{split}
S^2_{j+1}&=\int_{\{\tau< \bar t\}} \psi_{j+1}(x)\, d\mu(x)
=\int_{\{\tau< t_{j+1}\}} \psi_{j+1}(x)\, d\mu(x)+\int_{\{t_{j+1}<\tau< \bar t\}} \psi_{j+1}(x)\, d\mu(x)\\
&=\int_{\{\tau< t_{j+1}\}} \rho\psi_j(x)\, d\mu(x)+\int_{\{t_{j+1}<\tau< \bar t\}} \left( 1-\rho+\rho\psi_j(x) \right)\, d\mu(x)\\
&=\rho S^2_j + \big(1-\rho\big)\big(m(\bar t)-m(t_{j+1})\big)
\leq\rho S^2_j + (1-\rho) D \big(S^1_j-S^2_j\big)\,.
\end{split}
\end{equation}
Here $D=KL$, where $K$ is given by~(\ref{stimatS}) and $L$ is the Lipschitz constant of $m$: indeed, by immediate geometric arguments one understands that $t\mapsto \L^d \big( \big\{x\in\Omega:\, \tau(x)\leq t\big\} \big)$ is a Lipschitz map, so that $m(t)$, which is the integral over the sets $\big\{x\in\Omega:\, \tau(x)\leq t\big\}$ of the bounded function $f$, is also Lipschitz.\par
Now, using~(\ref{S1j+1}) and~(\ref{S2j+1}) we have
\[
S^1_{j+1}-S^2_{j+1} \geq \big(\rho-(1-\rho) D\big) \big(S^1_j-S^2_j\big)\,.
\]
We can finally deduce what follows: if the prudence is close enough to $1$, namely $\rho>1-\frac{1}{D+1}$, the sign of $S^1_m-S^2_m$ remains positive for all $m\geq j$. In other words, if the population is prudent enough and at a certain moment too many people are going to the service $x_1$, then in the sequel the same service will always have more people than in the equilibrium, though the excess becomes smaller and smaller.\par
We can now observe that $S^1_m$ and $S^2_m$ go to $0$ for $m\to \infty$. Indeed, by~(\ref{S1j+1}) we have that $S^1_m=\rho^{m-j} S^1_j$, so $S^1_m\to 0$; moreover, being $0\leq S^2_m\leq S^1_m$, it is clearly also $S^2_m\to 0$.\par
Using~(\ref{stimatS}) we deduce that $t_m\to \bar t$, and finally by~(\ref{evolutionprudence}) this implies that $\psi_m$ converges to $\bar\psi$ uniformly on the compact subsets of $\Omega\setminus\{\tau=\bar t\}$.
\end{proof}

\begin{example}[Non-convergence]\label{nonconv2}
{\rm We show that the claim of Theorem~\ref{thmprudence} is sharp, in the sense that a small prudence could be not enough to ensure convergence. As in the second part of Example~\ref{nonconv}, assume to be in a small city with two supermarkets having long queues, and take the prudence $\rho=1/3$. Assume also that $\psi_0(x)=1/4$ for $f-$a.e. $x$: then, $25\%$ of people are going to the first supermarket and $75\%$ to the second one. Since the city is small, the queue convinces \emph{all the citizens} that the first supermarket could be better; hence it will be $\psi_1(x)=3/4$ for $f-$a.e. $x$, so that at the day $1$ one has $75\%$ of people going to the first supermarket and $25\%$ to the second one. Clearly, this procedure will be periodic and there will be no convergence to the equilibrium.}
\end{example}

We conclude by noticing that this model with a ``fixed prudence'' may seem not so convincing, since it appears more likely that the prudence is increasing with the time: this increment can be also thought as a sign of ``laziness'', or of habit. We describe now how a model with an increasing prudence (that is, $\rho=\rho_j$ in the definition~\eqref{evolutionprudence}) gives rise to convergence to the equilibrium.

\begin{thm}[Increasing prudence]\label{pruddiver}
Assume that $f$ is bounded and that $h_1$ and $h_2$ are non-decreasing and Lipschitz continuous. Consider a sequence $(\rho_j)_{j \geq 1}$ such that
$$
\rho_j \to 1 \qquad \text{ and } \qquad \sum_{j=1}^\infty (1-\rho_j) = +\infty \,.
$$
Consider an evolution with increasing prudence, i.e.~consider \eqref{evolutionprudence} with $\rho = \rho_j$.
Then there is convergence of $\psi_j$ to $\bar\psi$ uniformly on the compact subsets of $\Omega\setminus\{\tau=\bar t\}$.
\end{thm}
\begin{proof}
Since $\rho_j\to 1$, there is $\bar j\in\N$ so that for all $j\geq \bar j$ one has $\rho_j \geq \bar \rho$, being $\bar\rho$ as in Theorem~\ref{thmprudence}. From the proof of that theorem, we know that, assuming by symmetry that $S^1_{\bar j} \geq S^2_{\bar j}$, for all $j\geq \bar j$ it is still $S^1_j \geq S^2_j$. Notice also that a trivial calculation ensures
\begin{equation}\label{prod->0}
\prod_{j = \bar j}^\infty \rho_j = \exp \bigg(\ln \Big(\prod_{j= \bar j}^\infty \rho_j \Big)\bigg)
=\exp \Big(\sum_{j=\bar j}^\infty \ln \rho_j \Big)
\leq \exp \Big(\sum_{j= \bar j}^\infty - (1-\rho_j) \Big)=0\,.
\end{equation}
Keeping in mind~(\ref{S1j+1}) and the fact that for all $j\geq \bar j$ it is $S^2_j\leq S^1_j$, we have that $S^1_j=S^1_{\bar j} \prod_{l={\bar j}}^{j-1} \rho_l\to 0$ and so also $S^2_j\to 0$; by~(\ref{stimatS}), which does not depend on $\rho_j$, we deduce that $t_j\to \bar t$. Finally, the required convergence of $\psi_j$ to $\bar\psi$ is shown exactly as in Theorem~\ref{thmprudence} recalling~(\ref{prod->0}).
\end{proof}

We now briefly discuss the assumptions of the above theorem: as already said, the meaning of $\rho_j\to 1$ is that the people become more and more prudent, or more and more lazy. The fact that $\sum_j (1-\rho_j)$ must diverge tells us that if the prudence goes to $1$ too fast, we can have convergence to a situation different from the equilibrium: for instance, if $\rho_j=1$ for all $j$, as already pointed out we have the full prudence, so that the people do never change their habit, and so $\psi_j=\psi_0$ for all $j$, whatever $\psi_0$ is. Finally, we have this last remark.

\begin{rem}\label{remsmart}
{\rm Notice that the constant $\bar\rho$ of Theorem~\ref{thmprudence} depends on the data of the problem. Since it is likely that the people only know how long the queues are everyday, but not the functions $f$, $h_1$ and $h_2$, a ``smart'' population does not have enough information to decide a fixed prudence $\rho$ which ensures the convergence to the equilibrium. However, thanks to Theorem~\ref{pruddiver}, the population can decide a strategy which surely leads to the convergence, that is an increasing prudence such as, for instance, $\rho_j=1-1/j$.}
\end{rem}

\subsection{Evolution with memory\label{lastsubs}}

We describe now the second possibility to overcome the unsatisfactory non-convergence behaviour of the simple evolution modeled in Section~\ref{secstand}. As anticipated, the idea is to allow people to have a memory, so that every day they can decide what to do keeping in mind more than just the queue of the day before. We go back to the deterministic model, in which the position where a citizen lives completely determines his behaviour. Recall that the function $G$ given by~(\ref{defG}) represents the difference between the queues at $x_2$ and at $x_1$ if a quantity $m(t)$ of people is going to $x_1$.\par
In the model of Section~\ref{secstand}, each citizen assumes that at the day $j+1$ the queue would have been the same as in the day $j$. Here, we give to the people a ``memory'': in other words, we take $\kappa\in \N$, $\kappa>1$, and in this model each citizen guesses that the queue will be approximatively the average of the queues of the preceding $\kappa$ days. Therefore, the citizen assumes that the difference between the queues at $x_2$ and at $x_1$ at the day $j+1$ will be
\[
Q=\frac{\sum_{m=j-\kappa+1}^{j} G(t_m)}{\kappa}\,.
\]
As a consequence, the citizen living at $x$ will decide to go to $x_1$ at the day $j+1$ if $\tau(x)<Q$. We take $\psi_1,\,\psi_2,\,\dots\,,\, \psi_\kappa$ as data of the problem, which means that for the first $\kappa$ days we let the people decide freely how to move (as in Section~\ref{secstand}, the choice of these $\psi_j$'s will not effect any of our results). For each $n>\kappa$, one has again
\[
\psi_n(x)=\left\{\begin{array}{ll}
0 &\hbox{if $\tau(x)<t_n$}\\
1 &\hbox{otherwise,}
\end{array}\right.
\]
but this time the evolution rule is simply
\begin{equation}\label{evolmemory}
t_{n+1}= \frac{\sum_{m=n-\kappa+1}^{n} G(t_m)}{\kappa}\,.
\end{equation}
We give now the first result of this section, which reduces to Proposition~\ref{g'<1} in the extreme case $\kappa=1$.

\begin{thm}\label{g'<L}
Assume that $h_1$ and $h_2$ are continuous and non-decreasing. Assume that the function $G$ given by~(\ref{defG}) is Lipschitz with constant $L<\kappa$. Then $t_n\to \bar t$ for $n\to +\infty$, hence also $\psi_n\to \bar\psi$ uniformly on the compact subsets of  $\Omega\setminus \{\tau=\bar t\}$.
\end{thm}

To prove this theorem, we will need the following definition.

\begin{definition}\label{defI}
Given $n\in \N$ and $j\in \{1,\,2,\, \dots\,,\,\kappa\}$, we say that the property $P_n^j(\alpha)$ holds true when, given any set $I\subseteq \{n,\, n-1,\, \dots n-\kappa+1\}$ with $\# I = j$ and $n\in I$, there holds
\begin{equation}\label{pmjalfa}
\bigg| \sum_{i\in I} \big[ G(t_i)-G(\bar t) \big] \bigg| \leq \alpha\,.
\end{equation}
\end{definition}

The fundamental feature of the properties $P_n^j(\alpha)$ is given by the lemma below.

\begin{lemma}\label{lemmachloe}
Assume that $G$ is $L-$Lipschitz with a constant $L\leq \kappa$. Then take $n\in\N$ and $\alpha>0$, and assume that $P_m^j(\alpha)$ holds true for any $j\in \{1,\,2,\, \dots\,,\,\kappa\}$ and for any $m\in \{n,\, n-1,\, \dots\,,\, n-\kappa+1\}$. We have then what follows.\\
{\bf (A)} For any $j\in \{1,\,2,\, \dots\,,\,\kappa\}$ and for any $m>n$, one has that $P_m^j(\alpha)$ is true.\\
{\bf (B)} For any $j\in \{1,\,2,\, \dots\,,\,\kappa\}$ and for any $m>n+(j-1)\kappa$ one has that $P_m^j(C_j\alpha)$ is true, where
\begin{equation}\label{defcj}
C_j = 1- \bigg(1-\frac{L}{\kappa}\bigg)^j\,.
\end{equation}
\end{lemma}

Before proving the lemma, let us notice how the result of Theorem~\ref{g'<L} easily follows.

\proofof{Theorem~\ref{g'<L}}
Let us set $n=\kappa$: there exists a positive $\alpha$ such that all the properties $P^j_m(\alpha)$ are valid for $j$ and $m$ in $\{1,\,2,\, \dots\,,\,\kappa\}$, since there are a finite number of inequalities of the type~(\ref{pmjalfa}) to be checked. By part~{\bf (A)} of Lemma~\ref{lemmachloe} we know then that all the properties $P^j_m(\alpha)$ hold true for any $m\in\N$ and $1\leq j\leq \kappa$. Moreover, since the constants $C_j$ are increasing in $j$ because $L<\kappa$, by part~{\bf (B)} we also have that the properties $P^j_m(C_\kappa \alpha)$ are valid for any $j$ and all $m>\kappa+ (\kappa-1) \kappa$, that is
\[
m \geq \kappa + \kappa(\kappa-1)+1\,.
\]
By an immediate iteration, we obtain the validity of the properties $P^j_m(C_\kappa^h\alpha)$ for all $j$ and all
\[
m\geq \kappa+ h\big(\kappa(\kappa-1)+1\big)\,.
\]
Since $C_\kappa<1$, the properties with $j=1$ are sufficient to deduce that $G(t_i)-G(\bar t)\to 0$ for $i\to \infty$, and since $G(t_i)=t_{i+1}$ while $G(\bar t)=\bar t$ this ensures that $t_i\to \bar t$ which is the first part of the thesis. The second part follows from the first one as already done in Proposition~\ref{g'<1}.
\end{proof}

Let us then prove the lemma.
\proofof{Lemma~\ref{lemmachloe}}
We start by proving {\bf (A)}. Let us then fix $\bar m>n$ and $\bar j$. We aim to show the validity of the property $P_{\bar m}^{\bar j}(\alpha)$. We can assume as an inductive hypothesis that $P_m^j(\alpha)$ has already been established 
\begin{itemize}
\item for all $n<m<\bar m$ and all $j$;
\item for $m=\bar m$ and all $j<\bar j$.
\end{itemize} 
Let us now fix a set $I$ according with Definition~\ref{defI}, and let us assume that $t_{\bar m}\leq \bar t$ (so that  $G(t_{\bar m})\geq G(\bar t)$), which is clearly admissible by symmetry. Keeping in mind that $G(\bar t)=\bar t$, and then
\begin{equation}\label{keepmind}
t_{\bar m} - \bar t = \frac{\sum_{\{\bar m-\kappa\leq i<\bar m\}}\ \left[ G(t_i)-G(\bar t) \right] }{\kappa}\,,
\end{equation}
we can first give the simple estimate
\begin{equation}\label{primaineq}
\sum_{i\in I} \left[ G(t_i) - G(\bar t) \right] = G(t_{\bar m})-G(\bar t) + \sum_{i\in I\setminus \{{\bar m}\}} \left[ G(t_i) - G(\bar t) \right]
\geq \sum_{i\in I\setminus \{{\bar m}\}} \left[ G(t_i) - G(\bar t) \right]  \geq -\alpha\,.
\end{equation}
Notice that the last inequality is emptily true for $\bar j=1$, even in the stronger form where $0$ replaces $-\alpha$, while for $\bar j>1$ it is a consequence of the property $P^{\bar j-1}_m(\alpha)$ with $m=\max \{i\in I\setminus  \{{\bar m}\}\}$, which is already known to hold by assumption.

We must now show also the opposite inequality, that is,
\begin{equation}\label{secondaineq}
\sum_{i\in I} \left[ G(t_i) - G(\bar t) \right] \leq \alpha\,.
\end{equation}
To this aim notice that, calling $J=\{\bar m-1,\, \bar m-2,\,\dots\,,\, \bar m -\kappa\}\setminus I$, one can evaluate
\begin{equation}\label{ultstr}\begin{split}
t_{\bar m}-\bar t &= \frac{\Big(\sum_{i\in I\setminus \{{\bar m}\}}\ \left[ G(t_i)-G(\bar t) \right] \Big)+\Big(\sum_{i\in J}\ \left[ G(t_i)-G(\bar t) \right] \Big)}{\kappa}\\
&\geq \frac{\Big(\sum_{i\in I\setminus \{{\bar m}\}}\ \left[ G(t_i)-G(\bar t) \right] \Big)-\alpha}{\kappa}\,,
\end{split}\end{equation}
using $P^j_{m'}(\alpha)$ with $m'=\max J<\bar m$ and $j=\# J$. Since in the same way one also has $\sum_{i\in I \setminus \{{\bar m}\}}\ \left[ G(t_i)-G(\bar t) \right] \leq \alpha$, the right term in last inequality is nonpositive: hence, by the $L-$Lipschitzianity of $G$ and the fact that $L\leq \kappa$ we obtain
\begin{equation}\label{stimagrezza}
G(t_{\bar m}) - G(\bar t) \leq L \, \frac{\alpha-\sum_{i\in I\setminus \{{\bar m}\}}\ \left[ G(t_i)-G(\bar t) \right] }{\kappa}
\leq \alpha-\sum_{i\in I\setminus \{{\bar m}\}}\ \left[ G(t_i)-G(\bar t) \right] \,,
\end{equation}
which gives~(\ref{secondaineq}). Hence, we have shown {\bf (A)} in its full generality.

\medskip

Let us now attack~{\bf (B)}, which will be done performing the same estimates as above, but more carefully. Our proof will be obtained as an induction over $\bar j$, so we must start with the case $\bar j=1$. In fact, this case is trivial: for any $\bar m>n$, recalling~(\ref{keepmind}) we have
\begin{equation}\label{nsc}
\big| t_{\bar m} - \bar t \big| = \Bigg |\frac{\sum_{\{\bar m-\kappa\leq i<\bar m\}}\ \left[ G(t_i)-G(\bar t) \right] }{\kappa} \Bigg|\leq \frac{\alpha}{\kappa}\,,
\end{equation}
and hence by the $L-$Lipschitzianity of $G$ one has
\[
\big| G(t_{\bar m})-G(\bar t)\big| \leq \frac{L}{\kappa}\, \alpha\,;
\]
this means that for any $\bar m>n$ one has the validity of $P^1_{\bar m}(C_1 \alpha)$ with $C_1=\frac{L}{\kappa}$, which coincides with~(\ref{defcj}).\par
We now take $1<\bar j\leq \kappa$, suppose by induction that {\bf (B)} has been proved for any $j<\bar j$, and try to obtain the thesis. We take then $\bar m>n+(\bar j-1)\kappa$, a set $I$ as in Definition~\ref{defI} and, as before, to fix the ideas we start assuming without loss of generality that $t_{\bar m}\leq \bar t$. Since, as a consequence, we have $G(t_{\bar m})\geq G(\bar t)$, it is immediate to give the lower bound. Indeed, exactly as in~(\ref{primaineq}), we can estimate
\begin{equation}\label{primaineq'}
\begin{split}
\sum_{i\in I} \left[ G(t_i) - G(\bar t) \right] &= G(t_{\bar m})-G(\bar t) + \sum_{i\in I\setminus \{{\bar m}\}} \left[ G(t_i) - G(\bar t) \right] \geq \sum_{i\in I\setminus \{{\bar m}\}} \left[ G(t_i) - G(\bar t) \right] \\
&\geq -C_{\bar j-1}\alpha
\geq -C_{\bar j} \alpha\,.
\end{split}
\end{equation}
To do so, we have used that $C_{\bar j-1}\leq C_{\bar j}$, which is clear by the definition~(\ref{defcj}), and we have also used the validity of $P^{\bar j-1}_{m'}(C_{\bar j-1}\alpha)$ where $m'=\max\{i\in I\setminus \{{\bar m}\}\}$. And in turn, $P^{\bar j-1}_{m'}(C_{\bar j-1}\alpha)$ is true by the inductive assumption and since
\[
m'>\bar m-\kappa >n+(\bar j-1)\kappa - \kappa = n + (\bar j-2)\kappa\,.
\]
To conclude the proof, we need then to show the upper bound, that is,
\begin{equation}\label{secondaineq'}
\sum_{i\in I} \left[ G(t_i) - G(\bar t) \right] \leq C_{\bar j}\alpha\,.
\end{equation}
We start exactly as we did to show~(\ref{secondaineq}) in the proof of {\bf (A)}; more precisely, we keep in mind the first inequality in~(\ref{stimagrezza}), which has already been proved above, that is
\[
G(t_{\bar m}) - G(\bar t) \leq L \, \frac{\alpha-\sum_{i\in I\setminus \{{\bar m}\}}\ \left[ G(t_i)-G(\bar t) \right] }{\kappa}\,.
\]
In order to use in a more careful way the $L$-Lipschitzianity of $G$, we rewrite the last inequality in the equivalent form
\[
G(t_{\bar m}) - G(\bar t) \leq \frac L \kappa \, \alpha - \frac L \kappa \, \sum_{i\in I\setminus \{{\bar m}\}}\ \left[ G(t_i)-G(\bar t)\right] \,,
\]
from which we obtain, using again some $P^{\bar j-1}_{m'}(C_{\bar j-1}\alpha)$ of which we already know the validity by inductive assumption, that
\begin{equation}\label{conclusion}\begin{split}
\sum_{i\in I} \left[ G(t_i) - G(\bar t) \right] &
\leq \frac L \kappa \, \alpha + \Big( 1- \frac L \kappa\Big) \, \sum_{i\in I\setminus \{{\bar m}\}}\ \left[ G(t_i)-G(\bar t) \right] \\
&\leq \frac L \kappa \, \alpha + \Big( 1- \frac L \kappa\Big) \, \Bigg| \sum_{i\in I\setminus \{{\bar m}\}}\ \Big[ G(t_i)-G(\bar t) \Big] \Bigg|
\leq \frac L\kappa \, \alpha + \Big( 1- \frac L \kappa\Big) \, C_{\bar j-1} \alpha\\
&= \bigg[ \frac L \kappa + \Big( 1- \frac L \kappa\Big) \bigg(1-\Big(1-\frac L \kappa\Big)^{\bar j -1}\bigg)\bigg] \alpha
= C_{\bar j} \alpha\,,
\end{split}\end{equation}
hence~(\ref{secondaineq'}) is established and we have completed the proof.
\end{proof}

As in the previous cases, we can give a counterexample to show that a small memory coefficient $\kappa$ could be not enough to guarantee convergence.

\begin{example}[Non-convergence]
{\rm This example is very similar to the ones in Examples~\ref{nonconv} and~\ref{nonconv2}: take a small city with the two supermarkets, and assume that the first supermarket can have a long queue, but the second one can have a much longer one. Take the memory coefficient $\kappa=2$, and assume that in the days $0$ and $1$ everybody goes to the first supermarket. This gives, for the days $j=0$ and $j=1$, an average queue which is high for the first supermarket and null for the second one: hence, at the day $j=2$ everybody will go to the second supermarket. The average queue for the days $j=1$ and $j=2$, then, is high for the first shop but very high for the second, thus everybody will go back to the first one at the day $j=3$. Again, the average queue for the days $j=2$ and $j=3$ will be high in the first supermarket and very high for the second one, so also at the day $j=4$ everybody will go to the first one. But then, at the day $5$ all the people will move to the second one because the average queue for the days $j=3$ and $j=4$ is all at the first one. This procedure is clearly periodic with a period of $3$ days and so the situation does not converge to the equilibrium. It is obvious how to modify the example to show that any memory coefficient $\kappa$ can be not enough for a suitable choice of $\Omega$, $f$, $h_1$ and $h_2$.}
\end{example}

Let us give also a ``precise'' counterexample to show the exact role of the assumption $L<\kappa$.

\begin{example}
{\rm Let us take a situation so that $G(t)=-2t$ for $-3\leq t\leq 3$. Assume also that $\kappa=2$ and $t_0=t_1=-1$. Then it is immediate to deduce by~(\ref{evolmemory}) that it will be $t_j=-1$ for all $j\equiv 0\ {\rm (mod\ 3)}$ or $j\equiv 1\ {\rm (mod\ 3)}$, while $t_j=2$ for all $j\equiv 2\ {\rm (mod\ 3)}$.}
\end{example}

In the example above there is no convergence even though the Lipschitz constant $L$ coincides with the memory coefficient $\kappa$, and it is immediate to modify the example to find a situation with no convergence with $L=\kappa$ for any $\kappa$. This ensures that the assumption $L<\kappa$ in Theorem~\ref{g'<L} is almost sharp. Indeed, we can slightly strengthen the assumption of Theorem~\ref{g'<L} covering also the situation in which $L=\kappa$ but the Lipschitzianity is strict. We recall that a function $\varphi:X\to Y$ is said \emph{strictly Lipschitz with constant $L$} if for any $x_1\neq x_2$ in $X$ one has
\[
\big\|\varphi(x_1)-\varphi(x_2)\big\|_Y < L \|x_1-x_2 \|_X\,.
\]
We have the following result.
\begin{thm}
Assume that $h_1$ and $h_2$ are continuous and non-decreasing and that the function $G$ given by~(\ref{defG}) is strictly Lipschitz with constant $\kappa$. Then $t_n\to \bar t$ for $n\to +\infty$, hence also $\psi_n\to \bar\psi$ uniformly on the compact subsets of  $\Omega\setminus \{\tau=\bar t\}$.
\end{thm}
\begin{proof}
We start by recalling that the result of Lemma~\ref{lemmachloe} has been proved also in the case $L=\kappa$; therefore, we can start as in the proof of Theorem~\ref{g'<L} finding a constant $\bar\alpha$ and knowing that all the properties $P^j_m(\bar\alpha)$ with $1\leq j\leq \kappa$ and $m\in\N$ are valid. However, for $L=\kappa$ all the constants $C_j$ given in~(\ref{defcj}) are all equal to $1$, so obtaining $P^j_m(C_j^h\bar\alpha)$ is of no help.\par
We now claim what follows: there exists a continuous and decreasing function $\xi:\R^+\to\R^+$ such that $\xi(\alpha)<\alpha$ for all $\alpha>0$, with the property that if all the $P^j_m(\alpha)$ with $1\leq j \leq \kappa$ and $m\geq n$ are valid, then all the $P^j_m(\xi(\alpha))$ with $1\leq j\leq \kappa$ and $m>n + \kappa(\kappa-1)$ hold. Notice that this will immediately give the thesis, because for any $h\in\N$ we will have, for $m$ big enough, all the properties $P^j_m(\xi^h(\alpha))$, and by the continuity of $\xi$ the sequence $h\mapsto \xi^h(\bar\alpha)$ must converge to $0$.\par
To conclude the proof we take $\alpha>0$ and we suitably modify the proof of Lemma~\ref{lemmachloe} for the case of $L=\kappa$ but with the strict Lipschitzianity, so to show the existence of the desired constant $\xi(\alpha)<\alpha$. Our plan is to find constants $\alpha_j<\alpha$ for $1\leq j\leq \kappa$ in such a way that $P^j_m(\alpha_j)$ is true for $m>\bar m + \kappa(j-1)$. The key idea is the following: since $G$ is \emph{strictly} $\kappa-$Lipschitz, for any $\eps \in ]0,\alpha[$ there exists $\kappa_\eps<\kappa$ so that
\begin{align}\label{kappaeps}
\big| G(t)-G(\bar t)\big| \leq \kappa_\eps |t-\bar t|\,, && \hbox{for all $t\in [\bar t-\alpha,\bar t-\eps]\cup [\bar t+\eps,\bar t+\alpha]$}\,:
\end{align}
this is trivial by compactness, but it will also be very useful for our purpose. We start arguing by induction. The case $j=1$ is very easy: indeed, we already noticed in~(\ref{nsc}) that for all $\bar m>n$ one has $| t_{\bar m} - \bar t |\leq \alpha/\kappa$. Then we can fix an arbitrary $\eps<\alpha/\kappa$. If $|t_{\bar m}-\bar t|<\eps$, then by Lipschitzianity of $G$ we have
\[
\big|G(t_{\bar m})-G(\bar t)\big| \leq \kappa \eps\,;
\]
on the other hand, if $|t_{\bar m}-\bar t|\geq \eps$, by making use of~(\ref{kappaeps}) we know that
\[
\big|G(t_{\bar m})-G(\bar t)\big| \leq \kappa_\eps |t_{\bar m}-\bar t| \leq \frac{\kappa_\eps}{\kappa}\,\alpha\,.
\]
Hence we obtain the case $j=1$ with $\alpha_1=\min\{\kappa_\eps\alpha/\kappa, \kappa\eps\}<\alpha$.\par
Consider now the case $1<j\leq \kappa$, and assume that the claim has been already proved up to $j-1$. Let us begin as in Lemma~\ref{lemmachloe}: supposing that $t_{\bar m}\leq \bar t$ by symmetry, as in~(\ref{primaineq'}) we have
\begin{equation}\label{below}\begin{split}
\sum_{i\in I} \big[ G(t_i) - G(\bar t) \big] &= G(t_{\bar m})-G(\bar t) + \sum_{i\in I\setminus \{{\bar m}\}} \big[ G(t_i) - G(\bar t) \big]\\
&\geq \sum_{i\in I\setminus \{{\bar m}\}} \big[ G(t_i) - G(\bar t) \big] \geq -\alpha_{j-1}
\end{split}\end{equation}
and the estimate from below is done. Concerning the estimate from above, there are again two cases to consider: fix an arbitrary
\[
\eps < \frac{\alpha-\alpha_{j-1}}{\kappa}\,.
\]
If $|t_{\bar m}-\bar t|<\eps$, then
\begin{equation}\label{above1}
\sum_{i\in I}\ \left[ G(t_i)-G(\bar t) \right] \leq \kappa\eps + \sum_{i\in I\setminus \{{\bar m}\}}\ \left[ G(t_i)-G(\bar t) \right] \leq \kappa\eps +\alpha_{j-1}<\alpha\,.
\end{equation}
On the other hand, if $|t_{\bar m}-\bar t|\geq \eps$, we can use~(\ref{kappaeps}) to derive, from~(\ref{ultstr}) which says
\[
t_{\bar m}-\bar t\geq \frac{\Big(\sum_{i\in I\setminus \{{\bar m}\}}\ \left[ G(t_i)-G(\bar t) \right] \Big)-\alpha}{\kappa}\,,
\]
the estimate
\[
G(t_{\bar m}) - G(\bar t) \leq \frac{\kappa_\eps}{\kappa}\,\bigg( \alpha-\Big(\sum_{i\in I\setminus \{{\bar m}\}}\ \left[ G(t_i)-G(\bar t) \right] \Big)\bigg)\,.
\]
Therefore, we mimic~(\ref{conclusion}) to find
\begin{equation}\label{above2}
\sum_{i\in I} \left[ G(t_i) - G(\bar t) \right] \leq \frac{\kappa_\eps}{\kappa}\,\alpha + \Big(1-\frac{\kappa_\eps}{\kappa}\Big)\sum_{i\in I\setminus \{{\bar m}\}}\ \left[ G(t_i)-G(\bar t) \right]
\leq \frac{\kappa_\eps}{\kappa}\,\alpha + \Big(1-\frac{\kappa_\eps}{\kappa}\Big)\alpha_{j-1}\,.
\end{equation}
Recalling~(\ref{below}), (\ref{above1}) and~(\ref{above2}), we obtain the inductive step of the claim by setting
\[
\alpha_j := \max\bigg\{\alpha_{j-1},\, \kappa\eps +\alpha_{j-1},\, \frac{\kappa_\eps}{\kappa}\,\alpha + \Big(1-\frac{\kappa_\eps}{\kappa}\Big)\alpha_{j-1}\bigg\} <\alpha\,.
\]

We finally set $\xi(\alpha)=\alpha_\kappa=\max_{j=1,\ldots , \kappa} \alpha_j$; since the continuity of this map $\xi$ is clear by the construction, the proof is obtained.
\end{proof}

We conclude the paper with a last result in the same philosophy of Theorem~\ref{pruddiver}, in which we considered a prudence increasing in time: we consider what happens with an increasing memory. More precisely, we focus on the case of ``global memory'', where the citizens remember \emph{all} the previous days and then the definition~(\ref{evolmemory}) must be replaced by
\begin{equation}\label{globalmemory}
t_{n+1}= \frac{\sum_{m=1}^{n} G(t_m)}{n}\,.
\end{equation}
We can show that there is always convergence with global memory, and this leads exactly to the same considerations as in Remark~\ref{remsmart}.
\begin{thm}[Global memory]
Assume that $h_1$ and $h_2$ are continuous and non-decreasing. Assume that the function $G$ is Lipschitz, and that the evolution is given by the global memory~(\ref{globalmemory}). Then $t_n\to \bar t$ and so $\psi_n\to \bar\psi$ uniformly on the compact subsets of $\Omega\setminus\{\tau=\bar t\}$.
\end{thm}
\begin{proof}
We can rewrite~(\ref{globalmemory}) in a useful way as
\[
t_{n+1}-\bar t = \Big(1-\frac 1 n\Big) \big(t_n-\bar t) + \frac {1}{n}\, \big(G(t_n)-G(\bar t)\big)\,.
\]
Take then $n\in\N$ and assume for simplicity that $t_n\geq\bar t$, so that $G(t_n)\leq G(\bar t)$: we can estimate
\begin{gather*}
t_{n+1}-\bar t \leq \Big(1-\frac 1 n\Big) \big(t_n-\bar t\big)\,,\\
t_{n+1}-\bar t \geq \frac {1}{n}\, \big(G(t_n)-G(\bar t)\big)\geq -\frac{L}{n}\,\big(t_n-\bar t\big)\,;
\end{gather*}
as a consequence, provided $n$ is big enough, regardless of the sign of $t_n-\bar t$ we have
\[
\big|t_{n+1}-\bar t \big| \leq \Big(1-\frac 1 n\Big) \big|t_n-\bar t\big|\,.
\]
Since $\Pi_{n=1}^{\infty} \big( 1-1/n \big) = 0$, this immediately leads to the thesis.
\end{proof}

\begin{rem}{\rm One can of course consider a lot of other ``schemes of memory''. For instance, one could have a global memory with different weights for older days, such as
\[
t_{n+1}= \sum_{m=1}^{n} \frac{1}{2^m(1-2^{-n})}\,G(t_m) \,,
\]
where the term $(1-2^{-n})$ is put to let the sum of the coefficients equal $1$. It is not hard to notice that our method of proof can be easily adapted to a generality of those ``schemes'': it is important that $t_{n+1}$ is given by a mean of terms $G(t_j)$ for some times $j\leq n$ and that the associated coefficients are decreasing as the days become older; that is,
\[
t_{n+1} = \sum_{m=1}^n C(n,m) G(t_j)
\]
with $0\leq C(n,m)\leq 1$, $\sum_{m=1}^n C(n,m) =1$ and $C(n,m)\leq C(n,m')$ for $m<m'$. In these cases one has the desired convergence provided $G$ is $L$-Lipschitz with a suitable constant $L$ as well as general convergence for a Lipschitz $G$ for reasonable global schemes of memory.
}\end{rem}

\end{document}